\documentclass[11pt, letterpaper]{amsart}

\usepackage[OT2,T1]{fontenc}
\usepackage{indentfirst}
\usepackage{amstext}
\usepackage{amsthm}
\usepackage{amsopn}
\usepackage{amsfonts}
\usepackage{amsmath}
\usepackage{latexsym}
\usepackage[francais,english]{babel}
\usepackage{amscd}
\usepackage{amssymb}
\usepackage{amsmath}
\usepackage[all,cmtip]{xy}
\usepackage{leftidx}
\usepackage{graphicx}
\usepackage{etoolbox}
\patchcmd{\section}{\normalfont\scshape\centering}{\normalfont\bfseries}{}{}
\patchcmd{\subsection}{-.5em}{.5em}{}{}
\renewenvironment{proof}{{\noindent\bfseries Proof.}}{}

\usepackage{hyperref}

\newtheorem{theo}{{Theorem}}[section]
\newtheorem{coro}[theo]{{Corollary}}
\newtheorem{lemma}[theo]{{Lemma}}
\newtheorem{prop}[theo]{Proposition}

\theoremstyle{definition}
\newtheorem{remark}[theo]{{Remark}}
\newtheorem{defn}[theo]{Definition}
\newtheorem{example}[theo]{Example}
\numberwithin{equation}{section}

\newtheorem{notation}[theo]{Notation}
\newtheorem{question}[theo]{Question}

\newcommand{\Br}{\mathrm{Br}}
\newcommand{\NS}{\mathrm{NS}}

\newcommand{\rk}{\mathrm{rank}}

\newcommand{\disc}{\mathrm{disc}}

\newcommand{\TT}{\mathrm{T}}
\newcommand{\Pic}{\mathrm{Pic}}

\newcommand{\QQ}{{\mathbf{Q}}}
\newcommand{\ZZ}{{\mathbf{Z}}}
\newcommand{\RR}{\mathbf{R}}
\newcommand{\CC}{\mathbf{C}}
\newcommand{\NN}{\mathbf{N}}

\DeclareFontEncoding{OT2}{}{} 
  \newcommand{\textcyr}[1]{%
    {\fontencoding{OT2}\fontfamily{wncyr}\fontseries{m}\fontshape{n}%
     \selectfont #1}}
\newcommand{\sha}{{\mbox{\textcyr{Sh}}}}

\begin{document}

\title{K3 surfaces with real or complex multiplication}
\author{Eva  Bayer-Fluckiger}
\email{eva.bayer@epfl.ch}
\address{EPFL-FSB-MATH, Station 8, 1015 Lausanne, Switzerland}

\author{Bert van Geemen}
\email{lambertus.vangeemen@unimi.it}
\address{Dipartimento di Matematica, Universit\`a di Milano,
Via Saldini 50, I-20133 Milano, Italia}

\author{Matthias Sch\"utt}

\email{schuett@math.uni-hannover.de}
\address{Institut f\"ur Algebraische Geometrie, Leibniz Universit\"at
  Hannover, Welfengarten 1, 30167 Hannover, Germany\newline
  Riemann Center for Geometry and Physics, Leibniz Universit\"at
  Hannover, Appelstrasse 2, 30167 Hannover, Germany}

\date{July 26, 2026}

\subjclass[2010]{14J28 (primary); 14C30, 14J42 (secondary)}
\keywords{K3 surface, Hodge structure, quadratic form, real multiplication, complex multiplication, hyperk\"ahler manifold}

\begin{abstract} Let $E$ be a totally real number field of degree $d$ and let $m \geqslant 3$ be an integer.
We show that if $md \leqslant 21$ then there exists an $(m-2)$-dimensional family of complex projective $K3$ surfaces
with real multiplication by $E$.
Analogous results are proved for CM number fields
and also for all known higher-dimensional hyperk\"ahler manifolds.

\end{abstract}

\maketitle


\selectlanguage{english}

\section{Introduction} Let $X$ be a complex projective $K3$ surface,
and let $T_X$ be its transcendental lattice. The algebra of Hodge endomorphisms of $T_X$ is
$A_X = {\rm End}_{\rm Hdg}(T_{X,\QQ})$ where $T_{X,\QQ}:=T_X\otimes_{\ZZ}\QQ$.
In \cite{Z}, Zarhin proved that $A_X$ is a totally real or CM number field.

It is natural to ask which number fields occur in this way; more precisely, what are the possibilities of the pairs
$(A_X,{\rm dim}(T_{X,\QQ}))$?

Several partial results are in the literature. Taelman \cite{T} proved that if $E$ is a CM field of degree $\leqslant 20$, then
there exists a complex projective $K3$ surface $X$ such that $A_X \simeq E$ and ${\rm dim}_E(T_{X,\QQ}) = 1$. For totally real fields,
several results were proved in \cite{G}, Elsenhans--Jahnel \cite{EJ 14}, \cite{EJ 16}, \cite{EJ 23} and recently in \cite{GS}.

The main aim of this paper is to prove

\medskip
\noindent
{THEOREM A.} {\it Let $E$ be a totally real number field of degree $d$ and let $m$ be an integer with $m \geqslant 3$
and $md \leqslant 21$.
Then there exists an $(m-2)$-dimensional family of
complex projective $K3$ surfaces such that a very general member $X$ has the properties
$A_X \simeq E$ and ${\rm dim}_E(T_{X,\QQ}) = m$.}

\medskip
\noindent
{THEOREM B.} {\it Let $E$ be a CM number field of degree $d$ and let $m$ be an integer with $m \geqslant 1$
and $md \leqslant 20$.

If $m \geqslant 2$, then
there exists  an $(m-1)$-dimensional family of
complex projective $K3$ surfaces such that a very general member $X$ has the properties
$A_X \simeq E$ and ${\rm dim}_E(T_{X,\QQ}) = m$.

If $m = 1$, then there exist infinitely 
many non-isomorphic
complex projective $K3$ surfaces $X$ such that  $A_X \simeq E$ and ${\rm dim}_E(T_{X,\QQ}) = 1$.}

\medskip
If $A_X$ is totally real, the surface $X$ is said to have {\it real multiplication};
the condition $m \geqslant 3$ in Theorem A  is necessary by \cite{G}, cf.\ Lemma \ref{lem:end},
hence Theorem A is optimal. It
shows that any totally real field $E$ of degree $d=2,3,\ldots,7$ arises as $A_X$ for a projective K3 surface $X$.
Notice that a totally real field of degree $d=8,9,10$ may still occur as the real subfield of a CM field acting on
$T_{X,\QQ}$, but it cannot be all of $A_X$ again by Lemma \ref{lem:end}.

The term {\it complex multiplication} is usually
reserved for  the case when $A_X$ is a CM field and ${\rm dim}_{A_X}(T_{X,\QQ}) = 1$,
since this is equivalent to the Mumford-Tate group of $X$ being abelian, see \cite{Z}.
However, several authors use the term complex
multiplication for surfaces $X$ such that $A_X$ is CM, also when ${\rm dim}_{A_X}(T_{X,\QQ}) > 1$.

We obtain analogous results for projective hyperk\"ahler (HK) manifolds, 
these have Hodge structures that are similar to the one of a K3 surface,
to the extent that all the  totally real fields
and CM fields which are a priori possible for degree reasons actually occur
(see Theorems \ref{thm:HK+RM} and \ref{thm:HK+CM} in Section \ref{RMCMHK}).


Before we comment on the content of the paper in more detail,
let us motivate it by briefly highlighting implications of  RM or CM.
In a nutshell, the most striking consequence is the following from \cite[\S 7]{G}:

If a complex manifold $X$ of dimension $2n$ has a Hodge substructure $V\subset H^{2k}(X,\ZZ)$ 
of K3-type with RM or CM by a field $K$, 
then  any non-zero $x\in K$ induces an isomorphism $\phi_x:V\rightarrow V$;
moreover, this isomorphism gives an integral Hodge class in $H^{2n}(X\times X,\ZZ)$ 
(cf.\ \cite[\S 2.6]{vG00} or \cite[\S 4.8]{GS}).
The Hodge conjecture asserts that this class is the class of an algebraic cycle on $X\times X$, 
but this is not known even for the case that $X$ is a K3 surface.  
(In fact, in the K3 case, the Hodge conjecture holds on $X\times X$ 
if and only if all RM and CM Hodge classes as above are given by algebraic cycles, compare \cite[Remark 7]{Charles}.) 
These cycles can sometimes be obtained explicitly from the graphs of certain rational
self-maps
(which may be of independent interest, cf.\ \cite[\S 6,7]{GS})
or from covering structures as in \cite[\S 4,5]{GS}.

On the arithmetic side, there is a consequence of RM or CM relying on the Tate conjecture:
Assume that the K3 surface $X$ is defined 
over some number field $L$; upon increasing $L$, if necessary,  the Galois action of Gal$(\bar L/L)$ should commute 
with the action of the RM or CM field $K$ on the \'etale cohomology of $X\otimes_L \bar L$.
This has strong implications on the associated $L$-series and zeta functions
and can also be used  to detect (or exclude) RM or CM, cf.\ \cite{EJ 14}, \cite{EJ 16}.

\subsection*{Strategy of proofs}

To explain the strategy of the proofs,
we observe that having RM or CM depends only on the rational, polarized Hodge structure $T_{X,\QQ}$.
Let $E$ be the algebra of Hodge endomorphisms of this Hodge structure.
Then $E$ is a (commutative) field, and the $\QQ$-vector space $T_{X,\QQ}$  has a structure of an $E$-vector space.
Zarhin showed that the field $E$ is either totally real or a CM field,
and that the quadratic form $q$ which defines the polarization has the (adjoint) property
\begin{eqnarray}
\label{eq:qq}
q(\alpha x,y) \,=\, q(x,\overline {\alpha} y)
\end{eqnarray}
for all $x,y \in T_{X,\QQ}$ and all $\alpha \in E$ (see Section \ref{ss:(C)}).
Here $\alpha \mapsto \overline \alpha$ is the complex conjugation if $E$ is CM, and the identity if $E$ is
totally real.
A $\QQ$-vector space $U$ with quadratic form $q$ has this property if and only if
there exists of a hermitian (if $E$ is CM) or quadratic (if $E$ is totally real)
form $h : W \times W \to E$  such that
 for all $x,y \in W$, we have
$$
q(x,y)\, =\, {\rm Tr}_{E/{\bf Q}}(h(x,y));
$$
see Lemma \ref{transfer lemma}. Here $W = U$ as a $\QQ$-vector space,
but we consider $W$ as an $E$-vector space.
To distinguish between the quadratic form (over $\QQ$) $U = (U,q)$ and the
hermitian form $W = (W,h)$ (over $E$), we use different notations for the vector spaces, even though they are the same; we also introduce
the notation
$$
U\, =\, \TT(W)
$$
for the above property, and we say that $U$ is the {\it transfer} of $W$ (from $E$ to $\QQ$); see
\S \ref{transfer section}.

Even if we are mainly interested in the case $U=T_{X,\QQ}$ it is now quite natural to consider more general
Hodge structures of K3 type for which analogous facts hold. In particular the quadratic form $q$ should have signature
$(2,r-2)$ where $r=\dim_\QQ U$. Moreover, the Hodge structure, of weight two,
on $U$ should have $\dim U^{2,0}=1$ and it should not have any non-trivial Hodge substructure. In case $E$ is
totally real, there is a further condition on the signature of the eigenspaces of the $E$-action on $U\otimes_\QQ\RR$,
which we discuss later (see Section \ref{ss:(C)}).

Let us fix a quadratic form $U = (U,q)$ over $\QQ$ and a totally real or CM number field $E$.
The first question is:

\noindent
(i) Is there a form $W$ over $E$ such that $U \simeq \TT(W)$?

We give a complete answer to question (i)  in Section \ref{characterization section} provided $E$ is CM, or $E$ is totally real and $m \geqslant 3$;
these are precisely the cases we need for our applications (the case $E$ totally real and $m = 1,2$ is more difficult, and
not completely known, see Kr\"uskemper \cite{K} and the references therein).

However, in general such a Hodge structure $U$ will not be Hodge isometric to $T_{X,\QQ}$ for a K3 surface $X$ (or more
generally, for a hyperk\"ahler manifold $X$). For this we introduce a quadratic
$\QQ$-vector space $(V,q_V)$, modelled on $H^2(X,\QQ)$,
and we require that $U = (U,q)$ is isometric to a direct summand of $V = (V,q_V)$. This raises the questions :

\noindent
(ii) Does there exist a (quadratic or hermitian) form $W$ over $E$ and a quadratic form $V'$ over $\QQ$ such that 
$$
V\, \simeq\, \TT(W) \oplus V'?
$$

\noindent
(iii) In case of a positive answer, characterize the forms $W$ with the above property.

The surjectivity of the period map for hyperk\"ahler manifolds allows us to find simple K3 type Hodge structures on
$U$ such that moreover $E$ is contained in their endomorphism algebra.
The Hodge structure on $V'$ will be trivial, in fact $V'=\Pic(X)\otimes_\ZZ\QQ$, and the Hodge structure on $V$
is determined by the one on $U$.

The proofs of Theorems A and B mainly rely on results concerning questions (i) and (ii).
The answers to these questions provide additional geometric information,
as illustrated by the following result:

\medskip
\noindent
{THEOREM  ${\mathrm A}'$.} {\it
Let $L$ be a lattice of rank $\rho$ and signature $(1,\rho-1)$ that allows a primitive embedding into the K3 lattice.
Let $d \geqslant 1$ be an odd integer,
and let $m \geqslant 3$ be such that $\rho + md = 22$.
Let $E$ be a totally real field of degree $d$.
Then there exists an $m-2$-dimensional
family of complex projective $K3$ 
surfaces such that a very general member $X$ is such that
\begin{itemize}
\item[$\bullet$]
 $A_X \simeq E$;

\item[$\bullet$] 
${\rm Pic}(X) \simeq L$. 
\end{itemize}}

\smallskip
This  is a strengthening of Theorem A (in the case of totally real fields of odd degree): it implies the existence
of a family of $K3$ surfaces with $A_X = E$, with the added information about the Picard lattice, see Section \ref{s:Pic} for details.

\subsection*{Outline of the paper}

In more detail, after covering some basics on Brauer groups and quadratic forms in Sections \ref{s:Br} - \ref{s:inv},
we discuss the transfer of quadratic and hermitian forms from number fields to $\QQ$ in Section \ref{s:transfer}.
Sections \ref{s:CM},  \ref{s:RM} are oriented towards answering questions (ii) and (iii);
combined with the fundamental facts about K3 surfaces, the results of these sections
are sufficient to prove Theorems A and B in \S  \ref{s:pf}. Section \ref{s:CM} relies on the results of \cite{B 23}.
In Section  \ref{s:RM} we mainly use the results of Sections \ref{s:qf}, \ref{s:inv} and \ref{s:transfer}.
However, in the case of codimension 1 (that is, $md = 21$), these methods do not suffice: instead, we apply
some results of Kr\"uskemper in \cite{K}.
The methods of \cite{K} are then further developed in the next sections of the paper.

To answer question (i), we need some more results on quadratic forms. We start with a short introduction
to Witt groups in Section \ref{Witt}. The aim of Section \ref{characterization section}
is to characterize the quadratic forms that can be obtained as $\TT(W)$
for some quadratic (or hermitian) form over $E$.
In order to do this, we adapt Kr\"uskemper's methods to our context;
the background on Witt groups enters here.

After reviewing the necessary theory of general Hodge structures of K3 type in Section \ref{ss:general},
we consider those associated to K3 surfaces in  Section \ref{s:K3} and to HK manifolds in Section \ref{s:HK}.
Based on this we give  the proof of Theorems A and B in Section \ref{s:pf}.
In Section \ref{Picard section}, where Theorem ${\mathrm A}'$ is proved,
we obtain results on lattices (and not just on the $\QQ$-vector spaces they generate).
The paper concludes with results, analogous to Theorems A and B,
for HK manifolds in Sections \ref{HK+RM}, \ref{HK+CM}
and for Mumford--Tate groups in Section \ref{ss:MT}.

\section{Brauer groups}
\label{s:Br}

The aim of this section is to give some reminders on Brauer groups that will be useful throughout the paper. Let $k$ be a field of characteristic not 2.
The Bauer group of $k$, denoted by ${\rm Br}(k)$, can be defined using central simple algebras or Galois cohomology (see for instance \cite{GS 06}, \cite{L} or \cite{S-local}). 
We use the additive notation for this group. In the sequel, we only need the subgroup 
${\rm Br}_2(k)$ consisting of 
$x \in {\rm Br}(k)$ such that $x + x = 0$. 
Typical elements of this group are given by (classes of) quaternion algebras.

\begin{defn}
Let $a,b \in k^{\times}$. The quaternion algebra $(a,b)$ is
the  associative $k$-algebra with basis $\{1, i, j, k\}$ 
where $i^2=a, j^2 = b$ and $ij = k =-ji$.
\end{defn}

\begin{example} 
Over $\RR$, the only quaternion algebras are the classical Hamilton quaternions (which form a division algebra)
and the real algebra of $2\times 2$ matrices (which is also called split).

Over $\CC$, there is only one quaternion algebra up to isomorphism, 
namely the complex algebra of $2\times 2$ matrices.
\end{example}

We only need to consider the case where $k$ is a number field or a local field. For these fields, the group ${\rm Br}_2(k)$ is generated by quaternion algebras.
Their arithmetic is well understood in terms of local and global class field theory, 
see for instance the presentation in \cite{S-local} or \cite{Voight}.
We summarize here the properties that are needed in the paper.

If $k$ is a local field, then ${\rm Br}_2(k)$ has two elements, and we identify it with $\ZZ/2\ZZ$. If $k$ is a number field, then the Brauer--Hasse--Noether theorem gives us the following exact sequence 
\[
0 \to \Br_2(k) \to\bigoplus_v  \Br_2(k_v) \to \ZZ/2\ZZ \to 0
\]
where the sum runs over all places $v$ of $k$.
In other words, the global elements are determined by their localizations, and have to satisfy a reciprocity condition.

Another way of looking at the quaternion algebras $(a,b)$ is in terms of the associated Hilbert symbols. These are also denoted by $(a,b)$ and have the following three properties : 
$(a,b)$ is symmetric, i.e.\ $(a,b)=(b,a)$, 
it is bilinear (in particular, $(a,b+b') = (a,b) + (a,b')$),  and it satisfies the relation $(a,-a) = 0$ for all $a \in k^{\times}$. See for instance \cite[Corollary 11.13]{Sch} for a list of useful properties, including these ones.

For $k = {\bf Q}_p$, the field of $p$-adic numbers, the symbols $(a,b)$ can be computed using \cite[Theorem 1, p.\ 39]{S}.
The following examples will be useful in the sequel.

\begin{example} 
\label{ex:calc}
(1) Let $a \in k^{\times}$. There is  the relation
$(a,a)=(-1,-1)$, and this is equivalent to $(-1,-a)=0$.
Indeed, we have $(a,a) = (a,-a) + (a,-1) = (a,-1)$ (because $(a,-a) = 0$). Therefore $(a,a)=(-1,-1)$ gives  
$(a,-1) = (-1,-1)$, or equivalently $(a,-1) + (-1,-1) = 0$ which is equivalent to $(-a,-1) = 0$.

(2)  Let $a \in k^{\times}$. Then the relation  $(a,3a)=(-2,-6)$ can be reformulated as $(-3,-2a)=0$. Indeed, start with
$(a,3a) = (a,-a) + (a,-3) = (a,-3)$. We have 
$(a,3a)=(-2,-6)$, hence 
$(a,-3)=(-2,-6)$ by the previous argument, equivalently $(a,-3) + (-2,-6) = 0$. But $(-2,-6) = (-2,2) + (-2,-3) = (-2,-3)$. 
Therefore we get
$(a,-3) + (-2,-3) = 0$, equivalently $(-2a,-3) = 0$.
\end{example}

Another useful property is phrased in terms of quadratic extensions of $k$.

\begin{lemma} 
Let $a,b \in k^{\times}$. Then $(a,b) = 0$ if and only if $a$ is a norm of the quadratic extension $k(\sqrt b)/k$. 
\end{lemma}

To put this to work, we denote be the discriminant of a number field $E$ by  $\Delta_E$,
the norm of an extension $E/E_0$ by ${\rm N}_{E/E_0}$, and likewise the trace by ${\rm Tr}_{E/E_0}$.

\begin{example} 
\label{ex:sum2}
Let $d$ be a square free integer, and set $E = Q(\sqrt d)$. We have the following equivalences:
 
$d$ is a sum of two squares $\iff$ $d \in {\rm N}_{\QQ (i)/{\QQ}}({\QQ (i)})$ $\iff$ $(-1,d) = 0$ $\iff$ $-1 \in {\rm N}_{E/{\QQ}}(E)$
$\iff$ $d \in {\rm N}_{E/{\QQ}}(E)$ $\iff$ $\Delta_E \in {\rm N}_{E/{\QQ}}(E)$.
\end{example}

\section{Quadratic forms}\label{quadratic}
\label{s:qf}

This section recalls some basic facts about quadratic forms over $\QQ$; we refer to \cite{S}, Chap. IV, \S 3 for details and proofs. We start
with some basic notions concerning quadratic forms over fields.

Let $k$ be a field of characteristic not 2. A {\it quadratic form} over $k$ is by definition a pair $V = (V,q)$, where $V$ is a finite
dimensional $k$-vector space and
$q: V \times V \to k$ is a non-degenerate symmetric bilinear form. Every quadratic form $V$
can be diagonalized: there exist $a_1,\dots,a_n \in k^{\times}$ such that $V$ is isomorphic to the diagonal quadratic
form $\langle a_1,\hdots,a_n \rangle$. The {\it determinant} of $V$ is by definition ${\rm det(}V) = \prod_i a_i$ in $k^{\times}/k^{\times 2}$.
The {\it Hasse invariant} of $V$ is $w(V) =  \sum_{i < j} (a_i,a_j)$ in ${\rm Br}_2(k)$, where $(a_i,a_j)$ is the class
of the quaternion algebra determined by $a_i$ and $a_j$. 
If $W$ is another quadratic form, then 
$$w(V \oplus W) = w(V) +
w(W) + ({\rm det}(V),{\rm det}(W)).
$$

 If $k = \QQ$, the {\it signature} of $V$, denoted by ${\rm sign}(V)$, is by definition the signature of the quadratic form $V \otimes_{\QQ} \RR$.

 Recall that, for all prime numbers $p$, we have ${\rm Br}_2({\bf Q}_p) \simeq {\bf Z}/2{\bf Z}$ while also  ${\rm Br}_2(\RR) \simeq \ZZ/2 \ZZ$; we
identify these groups with $\{0,1\}$.

\begin{theo} \label{S} {\rm (i)} Two  quadratic forms over $\QQ$ are isomorphic if and only if they have the same dimension, determinant, Hasse invariant and
signature.
All quadratic forms $V$ over $\QQ$ of signature $(r,s)$ satisfy
\begin{enumerate}
\item[(1)]
The sign of ${\rm det}(V)$ is $(-1)^s$.
\item[(2)]
The image of $w(V)$ in ${\rm Br}_2(\RR) = \ZZ/2 \ZZ$ is $s(s-1)/2$ ${\rm mod}\, 2$.
\item[(3)]
If ${\rm dim}(V) = 1$, or if ${\rm dim}(V) = 2$ and if $p$ is a prime number such that the image of ${\rm det}(V)$ in $\QQ_p^{\times}/\QQ_p^{\times 2}$ is $-1$, then
the image of $w(V)$ in ${\rm Br}_2(\QQ_p)$ is $0$.
\end{enumerate}

\noindent
{\rm (ii)} Conversely, if $D \in \QQ^{\times}/\QQ^{\times 2}$, $(r,s)$ and $w \in {\rm Br}_2(\QQ)$ satisfy the conditions {\rm (1)}, {\rm (2)} and
{\rm (3)}  above, then there exists a quadratic form over $\QQ$ with signature $(r,s)$, determinant $D$ and Hasse invariant $w$.

\end{theo}

\begin{proof} See \cite{S}, Chap. IV, \S 3, Corollary of Theorem 9 and Proposition 7.
\qed
\end{proof}

\begin{remark}
In particular, in dimension 1 and 2, Theorem \ref{S}  implies the following classification of quadratic forms over $\QQ$ 
by means of the following invariants:
\begin{enumerate}
\item
In rank $1$, the determinant suffices (and the signature, but the signature is the sign of the determinant here);
\item
In rank $2$, determinant and Hasse invariant suffice (as the signature is determined by the Hasse invariant over the real numbers).
\end{enumerate}
\end{remark}

\begin{lemma}\label{useful} Let $U$ and $V$ be quadratic forms, let ${\rm sign}(U) = (r',s')$
and ${\rm sign}(V) = (r,s)$  with $r'\leqslant r$, $s' \leqslant s$; assume that either

\begin{itemize}
\item[$\bullet$]
 ${\rm dim}(U) < {\rm dim}(V) - 2$, or

\item[$\bullet$]
${\rm dim}(U) = {\rm dim}(V) - 2$ and moreover for all prime numbers $p$ such that $w(U) \not = w(V)$ at $p$,
we have ${\rm det}(U) \not = - {\rm det}(V)$ in $\QQ_p^{\times}/\QQ_p^{\times 2}$.
\end{itemize}
Then there exists a quadratic form $V'$ such that $V \simeq U \oplus V'$.

\end{lemma}

\begin{proof}
Let $V'$ be a quadratic form over $\QQ$ such that ${\rm dim}(V') = {\rm dim}(V) -  {\rm dim}(U)$, that the determinant
of $V'$ is ${\rm det}(V)\cdot {\rm det}(U)$, that the signature of $V'$ is
$(r - r', s-s')$ and that $w(V') = w(V) + w(U) + ({\rm det}(V), {\rm det}(U) )$; we claim that this is possible by Theorem \ref{S} (ii).
Indeed, this is clear if ${\rm dim}(U) < {\rm dim}(V) - 2$. Suppose that ${\rm dim}(U) = {\rm dim}(V) - 2$ and let $p$ be a prime number such that
${\rm det}(U)  {\rm det}(V) = -1$ in $\QQ_p^{\times}/\QQ_p^{\times 2}$. Then $({\rm det}(V), {\rm det}(U) ) = ({\rm det}(V),- {\rm det}(V) ) = 0$ at $p$. Moreover,
the hypothesis implies that $w(U) = w(V)$ at $p$, hence we have $w(V) + w(U) + ({\rm det}(V), {\rm det}(U) ) = 0$. Therefore Theorem \ref{S} (ii) implies that it is possible to choose $V'$ with the desired invariants in this case as well. 

\medskip

Therefore the invariants of $V$ and of  $U \oplus V'$ coincide; applying Theorem \ref{S} (i)
we conclude that $V \simeq U \oplus V'$.
\qed

\end{proof}

\section{Integral and rational quadratic forms} 
\label{s:ir}

An integral quadratic form, or {\it lattice} is a pair $(L,q)$, where $L$ is a free $\ZZ$-module of finite rank and
$q : L \times L \to {\ZZ}$ is a non-degenerate symmetric bilinear form;
it is said to be {\it even} if $q(x,x)$ is an even integer for all $x \in L$. All the lattices
occurring in this paper are even; some important examples are the {\it hyperbolic plane} $H$, the {\it negative $E_8$-lattice},
denoted by $E_8$, as well as the negative $A_2$-lattice.

If $(L,q)$ is a lattice, then $(V,q) = (L\otimes_{\ZZ}{\QQ},q)$ is a quadratic form over $\QQ$; in this paper, most of the
work will be done with the rational quadratic forms induced by the lattices rather than the lattices themselves. If there
is no ambiguity, we will use the same notation for both: for instance, $H$ also denotes the hyperbolic plane over $\QQ$,
and if $a_1,\dots,a_n$ are integers, we use the notation $\langle a_1,\dots,a_n \rangle$ for the diagonal form both
over $\ZZ$ and $\QQ$,
but whenever necessary we will distinguish lattice and quadratic form explicitly, cf.\ Section \ref{ss:back}.

 Moreover, the following notation will be useful:

\begin{notation}\label{I} If $n \geqslant 1$ is an integer, we denote by $I_n$ the $n$-dimensional negative unit form $\langle -1,\dots,-1 \rangle$,
considered as quadratic form over $\QQ$.

\end{notation}

To pass from lattices to rational quadratic forms, the following lemma will often be used.

\begin{lemma}\label{iso iq} We have the following isomorphisms:

\smallskip
{\rm (i)} $E_8 \otimes_{\ZZ}{\QQ} \simeq I_8$, \ \ {\rm (ii)} $A_2 \otimes_{\ZZ}{\QQ} \simeq \langle -2, -6 \rangle$,
\ \  {\rm (iii)} $\langle -2,-2 \rangle \otimes_{\ZZ}{\QQ} \simeq \langle -1, -1 \rangle$.

\end{lemma}

\noindent
{\bf Proof.} (i) See for instance \cite{OM}, \S 106.

(ii) The quadratic form $A_2 \otimes_{\ZZ}{\QQ}$
represents $- 2$, hence it is isomorphic to $\langle -2,a \rangle$ for some $a \in \QQ$. Since ${\rm det}(A_2) = 3$, we obtain
$a = -6$. (iii) is proved by the same argument. The quadratic form $ \langle -1, -1 \rangle$ represents $-2$, hence it is isomorphic
over $\QQ$ to  $\langle -2,a \rangle$ for some $a \in \QQ$; since the determinant of $ \langle -1, -1 \rangle$ is 1, we have
$a = -2$.

\section{Some invariants}
\label{s:inv}

For the applications to $K3$ surfaces, the so-called ``$K3$ lattice'' plays an important role.
Set $\Lambda_{3,19} = H^3 \oplus E_8^2$,
where $H$ is the hyperbolic plane, and $E_8$ is the negative $E_8$-lattice.
If $X$ is a complex projective $K3$ surface, then $H^2(X,\ZZ)$, with its
intersection form, is a lattice isomorphic to $\Lambda_{3,19}$.

Set $V_{K3} = \Lambda_{3,19} \otimes_{\ZZ} \QQ$.
In this section, we record some well-known results concerning the invariants of $V_{K3} $, and of orthogonal sums of the
hyperbolic plane. Note that  $V_{K3} \simeq H^3 \oplus I_{16}$ (cf. Lemma \ref{iso iq} (i)).

\begin{lemma}\label{invariants}
\begin{enumerate}
\item[{\rm (i)}] ${\rm dim}(V_{K3}) = 22$.
\item[{\rm (ii)}] The signature of  $V_{K3}$ is $(3,19)$.
\item[{\rm (iii)}] ${\rm det}(V_{K3}) = -1$.
\item[{\rm (iv)}] $w(V_{K3}) = (-1,-1)$, i.e.\ the Hasse invariant of $V_{K3}$ at a prime $p$ is 0 if $p \not = 2$, it is $1$ if $p = 2$ and at infinity.
\end{enumerate}

\end{lemma}

\noindent
{\bf Proof.} Statements (i)-(iii) are clear, and (iv) is proved by using the definitions of \S \ref{quadratic}. 
\qed

\begin{lemma} 
\label{lem:H^n}
Let $n \geqslant 1$ be an integer; we denote by $H^n$ the orthogonal sum of $n$ copies of the hyperbolic plane $H$. We have
\begin{enumerate}
\item
${\rm det} (H^n) = (-1)^n$.

\item 
$w(H^n) = 0$ at $p$ if $p$ is a prime number $p \not = 2$.

\item
At $p = 2$, we have $w(H^n) = 0$ if $n \equiv \ 0, 1  \ {\rm (mod \ 4)}$ and $w(H^n) = 1$ if
$n \equiv \ 2, 3  \ {\rm (mod \ 4)}$.
\end{enumerate}
\end{lemma}

\section{Hermitian forms and transfer}\label{hermitian}

Let $E$ be an algebraic number field, and let $e \mapsto \overline e$ be a $\QQ$-linear involution, possibly the identity. Let $E_0$
be the fixed field of the involution; $E_0 = E$ if the involution is the identity, otherwise $E/E_0$ is a quadratic extension. A {\it hermitian
form} is a pair $(W,h)$, where $W$ is a finite dimensional $E$-vector space, and $h : W \times W \to E$ is a sesquilinear form
such that $\overline {h(x,y)} = h(y,x)$ for all $x,y \in W$; if the involution is the identity, then $(W,h)$ is a quadratic form over $E$.

Every hermitian form over $E$ can be diagonalized, i.e.\ $(W,h) \simeq \langle \alpha_1,
\dots,\alpha_n\rangle$ for some $\alpha_i \in E_0^{\times}$.
The {\it determinant} of $W = (W,h)$ is by definition the product
$\alpha_1 \cdots \alpha_n$ considered as an element of $E_0^{\times}/{\rm N}_{E/E_0}(E^{\times})$
if the involution is non-trivial,
and of $E^{\times}/E^{\times 2}$ if it is the identity.

\begin{lemma}\label{transfer lemma} Let $(U,q)$ be a quadratic form over $\QQ$,
and suppose that $U$ has a structure of $E$-vector space. The following are equivalent:

\begin{enumerate}
\item
For all $x,y \in U$ and all $\alpha \in E$, we  have
$$
q(\alpha x,y)\, =\, q(x,\overline {\alpha} y).
$$

\item
There exists a hermitian form $h : U \times U \to E$ such that for all $x,y \in U$, we have
$$
q(x,y)\, =\, {\rm Tr}_{E/{\bf Q}}(h(x,y)).
$$
\end{enumerate}
\end{lemma}

\noindent
{\bf Proof.} (ii) $\implies$ (i) follows from the sesquilinearity of $h$; indeed, for all $x,y \in U$ and all $\alpha \in E^{\times}$, we  have 
$$q(\alpha x,y) = {\rm Tr}_{E/{\bf Q}}(h(\alpha x,y)) = {\rm Tr}_{E/{\bf Q}}(h(x, \overline {\alpha}y)) =  q(x,\overline {\alpha} y).$$

We next show that (i) implies (ii). Let us fix $x,y \in U$, and consider the $\QQ$-linear map $\ell : E \to {\QQ}$ defined by 
$\ell(\alpha) = q(\alpha x,y)$. Note that the quadratic form ${\rm Tr}_{E/{\QQ}} : E \times E \to \QQ$  given by 
${\rm Tr}_{E/{\QQ}} (\alpha,\beta) = {\rm Tr}_{E/{\QQ}}(\alpha\cdot\beta)$  is non-degenerate,
hence there exists a unique
$\beta \in E$ such that  ${\rm Tr}_{E/{\QQ}}(\alpha \beta) = \ell(\alpha)$ for all $\alpha \in E$. Set $h(x,y) = \beta$. It
is straightforward to check that this defines a hermitian form $h : U \times U \to E$. 
\qed

\section{Transfer}\label{transfer section}
\label{s:transfer}

An algebraic number field $E$ is said to be a {\it totally real field} if for all embeddings $E \to \CC$ we have $\sigma(E) \subset  \RR$; it is said to
be a {\it CM field} if it is a totally imaginary field that is a quadratic extension of a totally real field.
Recall that we introduced  $\Delta_E$ to denote the discriminant of $E$.

Let $E$ be a totally real or CM  field of degree $d$; if $E$ is CM, we denote by $x \mapsto \overline x$ the complex conjugation,
by $E_0$ the  maximal totally real subfield of $E$, and set $d_0 = [E_0:\QQ]$ (hence $d_0 = {\frac d2}$).

Let $W$ be a finite dimensional $E$-vector space, and let $Q : W \times W \to E$ be a quadratic form if $E$ is totally real, and
a hermitian form with respect to the complex conjugation if $E$ is a CM field.

We denote by $\TT(W)=(W,q)$ the quadratic form over $\QQ$ defined by
$$
q:\,W \times W\, \longrightarrow \QQ,\qquad q(x,y) ={\rm Tr}_{E/\QQ}(Q(x,y)),
$$
called the {\it transfer} of $W$, or more precisely of $(W,Q)$.

\begin{lemma}\label{invariants bis}  
{\rm (i)} ${\rm dim}_{\QQ}(\TT(W)) = d.{\rm dim}_E(W)$.

{\rm (ii)} If $E$ is totally real, then  ${\rm det}(\TT(W)) = \Delta_E^{{\rm dim}_E(W)}{\mathrm N}_{E/\QQ}({\rm det}(W))$.

{\rm (iii)} If $E$ is a CM field, then  ${\rm det}(\TT(W)) = [(-1)^{d_0} \Delta_E]^{{\rm dim}_E(W)}$.
\end{lemma}

\begin{proof} (i) is clear,  (ii) is proved in \cite{G}, Lemma 4.5 (i) and part (iii) follows from \cite{B 23}, \S 19, determinant condition.
\qed
\end{proof}

\medskip
 Suppose now that $E$ is a CM field.

\begin{notation} 
\label{notation}
We denote by $S_E$ the set of prime numbers $p$ such that  we have an isomorphism of
${\bf Q}_p$-algebras
$$
E \otimes_{\QQ} \QQ_p  \simeq
E_0 \otimes_{\QQ} \QQ_p \times E_0 \otimes_{\QQ} \QQ_p.$$
\end{notation}

In the above setting, $E \otimes_{\QQ} \QQ_p$ is called a \emph{split algebra}.
In the next section we will use that the discriminant  of a split algebra is trivial.

\begin{example} Let $E$ be a cyclotomic field, $E = {\bf Q}(\zeta_n)$. Then $S_E$ is the set of prime numbers $p$ such that
the subgroup of $({\bf Z}/n {\bf Z})^{\times}$ generated by $p$ does not contain $-1$ (see \cite{B 24} Proposition 5.4, \cite{B 23} Proposition 31.2). 
If $n$ is a prime number with $n  \equiv 3 \ {\rm (mod \ 4)}$, then $p \in S_E$ if
and only if $p$ is a square mod $n$ (see \cite{B 23}, Corollary 31.3). 

\end{example}

\begin{lemma}\label{w CM}
If $p \in S_E$, then for any hermitian form $W$ over $E$ we have
$$\TT(W) \otimes_{\QQ} \QQ_p \simeq (H \otimes_{\QQ} \QQ_p )^{d_0{\rm dim}_E(W)}.$$

\end{lemma}

\noindent
{\bf Proof.} See \cite{B 23}, \S 19, hyperbolicity condition.
\qed

\medskip

We now discuss the signature of  the forms $\TT(W)$  as above.
If $E$ is totally real, and $W = \langle a \rangle$ with $a \in E^{\times}$, then the signature of $\TT(W)$
is $(a^+,a^-)$, where $a^+$ is the number of real embeddings of $E$ where $a$ is positive and $a^-$ is the number of those where
$a$ is negative.

If $E$ is a CM field, and $W = \langle a \rangle$ with $a \in E_0^{\times}$, then the signature of $\TT(W)$
is $(2a^+,2a^-)$, where $a^+$ is the number of real embeddings of $E_0$ where $a$ is positive and $a^-$ is the number of those where
$a$ is negative.

Since quadratic and hermitian forms over $E$ are diagonalizable, this determines the signature of $\TT(W)$ for any $W$.

\begin{theo}
\label{realization} 

Suppose that $E$ is a CM field, and let $U$ be a quadratic form over $\QQ$. There exists 
a hermitian form $W$ over $E$ such that $U \simeq \TT(W)$ if and only if the following conditions hold : 

\begin{enumerate}
\item[{\rm (i)}]
 ${\rm dim}_{\QQ}(U) = m[E:\QQ]$.
 \smallskip

\item[{\rm (ii)}] ${\rm det}(U) = [(-1)^{d_0}\Delta_E]^{m}$.
 \smallskip

\item[{\rm (iii)}]
If $p \in S_E$, then
$U \otimes_{\QQ} \QQ_p \simeq (H \otimes_{\QQ} \QQ_p )^{d_0m}$.
 \smallskip

\item[{\rm (iv)}]
 The signature of $U$ is of the form $(2a,2b)$ for some integers $a,b \geqslant 0$.
\end{enumerate}
\end{theo}

\noindent
{\bf Proof.} If there exists a hermitian form $W$ over $E$ such that $U \simeq \TT(W)$ then (i) clearly holds, property (ii)
follows from Lemma \ref{invariants bis}  (ii), property (iii) from Lemma \ref{w CM} and property (iv) from the above discussion.
Conversely, suppose that conditions (i)-(iv) hold. The existence of a hermitian form $W$ over $E$
such that $U \simeq \TT(W)$
follows from \cite{B 23}, Theorem 17.2. Indeed, since $E$ is a field, we have $\sha_E = 0$.
The hypotheses
imply that condition (L 1) holds (see \cite {B 23}, Proposition 19.3). Therefore \cite{B 23}, Theorem 17.2 implies that
there exists a hermitian form $W$ over $E$ such that $U \simeq \TT(W)$.
\qed

%
%
%
%
%

\section{CM fields}
\label{s:CM}

The results of this section and the next one will be central for the proofs of Theorems A and B. We start with CM fields,
that is, the ingredients needed for the proof of Theorem B.

Let $E$ be a CM number field of degree $d$, and let
$m$ be an integer with $m \geqslant 1$.
We  start with a general result concerning the  case of  codimension greater than two
before covering the codimension two case  in the K3 setting (Proposition \ref{cm 2}).

\begin{theo}\label{cm Hodge} Let $V$ be a quadratic form over $\QQ$ of signature $(r,s)$,
 and let $m \geqslant 1$ be an integer with $ {\rm dim}(V)  > md - 2$. 
Let  $r',s' \geqslant 0$ be integers such that $r' \leqslant r$, $s'\leqslant s$, and $r'+s' = md$.

 Let
$W$ be a hermitian form of dimension $m$ over $E$ such that the signature of $\TT(W)$ is $(r',s')$.
Then there exists a quadratic form $V'$ over $\QQ$ such that
$$
V \simeq {\TT}(W) \oplus V'.
$$
\end{theo}

\noindent
{\bf Proof.} This follows from Lemma \ref{useful}. 
\qed

\begin{coro}\label{cm 1} Let $V$ be a quadratic form over $\QQ$, and let $(r,s)$ be the
signature of $V$. 
Let $m$ be an integer with $m \geqslant 1$ such that ${\rm dim}(V) - md > 2$.
Let $r',s' \geqslant 0$ be even integers such that $r' \leqslant r$, $s'\leqslant s$, and $r'+s' = md$.

 Then there exists a hermitian form $W$ over $E$ such that $\TT(W)$ has signature $(r',s')$
and a quadratic form $V'$ over $\QQ$ such that
$$
V \simeq {\TT}(W) \oplus V'~.
$$

\end{coro}

\noindent
{\bf Proof.}
Let $a_1,\dots,a_m \in E_0^{\times}$, and let $W$ be the hermitian form over $E$ with respect to the complex conjugation
defined by $W  = \langle a_1,\dots,a_m \rangle$; choose
$a_1,\dots,a_m$ such that ${\rm sign}(\TT(W)) = (r',s')$. This is possible since $r'$ and $s'$ are even; see the discussion
on signatures of the previous section.
By Theorem \ref{cm Hodge} there exists a quadratic form $V'$ such
that $V \simeq \TT(W) \oplus V'$.
\qed

\medskip

To cover the codimension two case, we continue with results that are specific to the $K3$ setting. 

\begin{theo}\label{cm Hodge bis} Let $V = V_{K3}$ and let $m$ be an integer with $m \geqslant 1$ such that $md =  20$. 
Set $\Delta_E' = [(-1)^{d_0} \Delta_E]^m$.
Let $a \in \QQ^{\times}$ and set $V' = \langle a,-a \Delta_E' \rangle$. 
Then there exists a hermitian form $W$ over $E$ such that the signature of $\TT(W)$ is $(2,18)$ and that

$$
V \simeq {\TT}(W) \oplus V'.
$$

\end{theo}

\noindent
{\bf Proof.} Note that the signature of $V'$ is $(1,1)$. Let $U$ be a quadratic form over $\QQ$ such that $V \simeq U \oplus V'$; this
is possible by Lemma \ref{useful}. The signature of $U$ is $(2,18)$. We have ${\rm det}(V') = - \Delta_E'$, hence ${\rm det}(U) = \Delta_E'$. 
This implies that $U$ satisfies condition (ii) of Theorem \ref{realization}; conditions (i) and (iv) clearly hold, so it remains to check
condition (iii).

We have $w(V) = w(U) + w(V') + ({\rm det}(U),{\rm det}(V'))$, and ${\rm det}(U) = \Delta_E'$, ${\rm det}(V') = - \Delta_E'$, therefore
$({\rm det}(U),{\rm det}(V')) = (\Delta_E', -\Delta_E') = 0$. This implies that $w(V) = w(U) + w(V')$.

We have $w(V') = (a,-a\Delta_E) = (a,-a) + (a,\Delta_E') = (a,\Delta_E')$.

If $m$ is even, then $\Delta_E'$ is a square, hence $w(V') = 0$. 

 Suppose that $m$ is odd; since $md = 20$,  we have either $m = 1$ and $d = 20$ or $m = 5$ and $d = 4$, therefore
$\Delta_E' = \Delta_E$. 
If $p$ is a prime number such that $p \in S_E$, then  $\Delta_E = 1$ in $\QQ_p^\times/{\QQ_p^\times}^2$
by the comment just after Notation \ref{notation}, hence 
$w(V') = (a,\Delta_E) = 0$.

Hence in both cases $w(U) = w(V) = (-1,-1)$ at all $p \in S_E$ by Lemma \ref{invariants}.
Here $w(U) = (-1,-1)$, because $w(V) = w(U) + w(V') +
(\det(U), \det(V'))$ and we already saw that this last term is trivial.
Lemma \ref{lem:H^n} implies that $w(H^{10}) = (-1,-1)$. Hence  $w(U) = w(H^{10})$ at $p$. 
Since $U$ and $H^{10}$ have the same dimension and determinant, 
this implies that they are isomorphic. Therefore $U$ satisfies condition (iii) of Theorem  \ref{realization}.

In summary, $U$ satisfies all the
conditions of Theorem \ref{realization}; hence there exists a hermitian form $W$ such that $U \simeq \TT (W)$. 
 \qed

\begin{prop}\label{cm 2} Let $V = V_{K3}$ and let $m$ be an integer with $m \geqslant 1$ such that $md \leqslant 20$.
Then there exists a hermitian form $W$ over $E$ such that $\TT(W)$ has signature $(2,md-2)$ and a quadratic form $V'$ over $\QQ$ such
that $$V \simeq {\TT}(W) \oplus V'.$$

Moreover, if $md = 20$ and $\Delta_E^m$ is a square, then $V' \simeq H$.

\end{prop}

\begin{proof} The existence part of the proposition follows from Theorem \ref{cm 1} if 
 $md < 20$, this follows from Theorem \ref{cm 1}, and  from Theorem \ref{cm Hodge bis} if  $md = 20$.

Suppose  that $md = 20$ and that $\Delta_E^m$ is a square. If $V = U \oplus V'$ with $U = \TT(W)$ for some hermitian form $W$, then have
${\rm det}(U) = ((-1)^{d_0}\Delta_E)^m$.

Note that we have ${\rm det}(U) = \Delta_E^m$. This is clear if $m$ is even, and
the hypothesis $md = 20$ implies that if $m$ is odd, then $d_0$ is even, hence ${\rm det}(U) = \Delta_E^m$ in this case as well.

 Since $\Delta_E^m$ is a square by hypothesis,  this implies that ${\rm det}(U) = 1$, therefore ${\rm det}(V') = -1$. As ${\rm dim}(V') = 2$, this 
implies that $V' \simeq H$. 
\qed
\end{proof}

\begin{prop}\label{m = 1} Let $V = V_{K3}$, suppose that $m = 1$ and $d \leqslant 20$. If $d = 20$, assume moreover that
$\Delta_E$ is not a square. Then there exist infinitely many non-isomorphic quadratic forms $U$ over $\QQ$ such
that 
\begin{itemize}
\item[$\bullet$]
there exists a one-dimensional hermitian form $W$ over $E$ with $U \simeq \TT(W)$;

\item[$\bullet$]
there exists a quadratic form $V'$ over $\QQ$ such that $U \oplus V' \simeq V$. 
\end{itemize}
\end{prop}

\noindent
{\bf Proof.} Suppose first that $d < 20$. We consider sets $\Sigma$  of finite places of $E_0$ having the following
properties
\begin{itemize}
\item[$\bullet$]
If $v \in \Sigma$, then $v$ is inert in $E$,

\item[$\bullet$]
The cardinality of $\Sigma$ is even,

\item[$\bullet$]
If $v \in \Sigma$ is above the prime number $p$ and $v' \in \Sigma$ with $v' \not = v$ is above the prime number $p'$,
then $p \not = p'$.
\end{itemize}

For a set $\Sigma$ as above, let $P(\Sigma)$ be the set of prime numbers such that for all $p \in P(\Sigma)$ there
exists a  $v \in \Sigma$ above $p$.  Note that such a $v$ is unique.

We shall use that there are infinitely many distinct sets of prime numbers $P(\Sigma)$ with $\Sigma$ as above.
Indeed, since $E/E_0$ is a quadratic extension, 
there are infinitely many places $v$ of $E_0$ that are inert in $E$. 
It is easy to see that the two other conditions can also be satisfied, 
so that we obtain infinitely many different sets $\Sigma$, hence also $P(\Sigma)$ as stated.

We show that if $\Sigma$ is a set as above, then there exists 
a quadratic form $U(\Sigma)$ over $\QQ$ with the desired properties, and that if $P(\Sigma') \not = P(\Sigma)$, then $U(\Sigma')$
and $U(\Sigma)$ are not isomorphic.

Fix  $\theta \in E_0$  such that $E = E_0(\sqrt{\theta})$. If $\Sigma$ is a set as above, let $\lambda(\Sigma) \in E_0^{\times}$ be such that 
$(\lambda(\Sigma),\theta) = 0$
at a place $v$ of $E_0$ if and only if $v \not \in \Sigma$; such a $\lambda(\Sigma)$ exists by reciprocity, see for
instance \cite{OM}, 72.19.

 Set $W(\Sigma)  = \langle \lambda(\Sigma) \rangle$ and $U(\Sigma) = \TT (W(\Sigma))$. 
Theorem \ref{cm Hodge}
implies that 
there exists a quadratic form $V'$ over $\QQ$ such that $U(\Sigma) \oplus V' \simeq V$.

Let $W_0$ be the unit form $W_0 = \langle 1 \rangle$ over $E$, and
set $U_0 = \TT (W_0)$. For all prime numbers $p$ we have
$$w(U(\Sigma)\otimes_{\QQ} {\QQ}_p) = w(U_0 \otimes_{\QQ} {\QQ}_p) + \underset{w|p} \sum {\rm cor}_{(E_0)_w/{\QQ}_p}(\lambda(\Sigma),\theta)$$
where the sum runs over the places $w$ of $E_0$ above $p$, the field $(E_0)_w$ is the completion of $E_0$ at $w$, and ${\rm cor}_{(E_0)_w/{\QQ}_p}$ is the corestriction ${\rm Br}_2((E_0)_w) \to {\rm Br}_2(\QQ_p)$
(see for instance \cite{B 23}, Proposition 12.4).

Let $\Sigma$ and $\Sigma'$ be two sets as above with $P(\Sigma) \not = P(\Sigma')$. Take  $p \in P(\Sigma)$ with $p \not \in P(\Sigma')$. 
By construction, we have $ \sum_{w\mid p} {\rm cor}_{(E_0)_w/{\QQ}_p}(\lambda(\Sigma),\theta) \not = 0$ 
(since  $(\lambda(\Sigma),\theta) = 0$ at $v$ if and only if $v \not \in \Sigma$)
and
 $\sum_{w\mid p} {\rm cor}_{(E_0)_w/{\QQ}_p}(\lambda(\Sigma'),\theta) = 0$, 
 hence $w(U(\Sigma)\otimes_{\QQ} {\QQ}_p) \not =
 w(U(\Sigma')\otimes_{\QQ} {\QQ}_p)$. This implies that 
 $U(\Sigma')$
and $U(\Sigma)$ are not isomorphic. 
This concludes the proof of the proposition in case $d<20$.


\smallskip
Suppose now that $d = 20$, let $a \in \QQ^{\times}$, and set $V' = \langle a,-a \Delta_E \rangle$. 
Then by Theorem  \ref{cm Hodge bis} there exists a hermitian form $W$ over $E$ such that the signature of $\TT(W)$ is $(2,18)$ and that
$
V \simeq {\TT}(W) \oplus V'$. We have $w(V') = (a,\Delta_E)$; since $\Delta_E \not = 1$, by varying $a$ we obtain infinitely many
non-isomorphic quadratic forms $V'$ with the required property, hence infinitely many non-isomorphic quadratic forms $U$. 
\qed

\medskip
Note that Proposition \ref{cm 2} implies the following uniqueness result:

\begin{prop}\label{m = Delta = 1} Let $V = V_{K3}$, suppose that $d = 20$ and that
$\Delta_E$ is  a square. Suppose that $W$ is a one-dimensional hermitian form over $E$ such
that $\TT(W)  \oplus V' \simeq V$. Then $V' \simeq H$. 

\end{prop}

We continue by deriving analogous results geared towards HK manifolds (see Section \ref{s:HK}).

\begin{prop}\label{Kum and K3[n] CM Hodge} Let $k > 0$ be an integer, and let $V = H^3  \oplus \langle -2k \rangle$
and $m$ be an integer such that $md < 6$,  or $V = H^3 \oplus I_{16} \oplus \langle -2k \rangle$ and $m$ be an integer such that $md < 22$. 
Let $W$ be a hermitian form over $E$ of dimension $m$ such that ${\rm sign}(T(W)) = (2,md-2)$. 
Then there exists
a quadratic form $V'$ over ${\bf Q}$ such that
$$
V \simeq {\TT}(W) \oplus V'~.
$$

\end{prop}

\noindent
{\bf Proof.} This follows from Theorem \ref{cm Hodge}.
\qed

\begin{prop}\label{Kum and K3[n] CM} Let $E$ be a CM field of degree $d$, let $k > 0$ be an integer, and let $V = H^3  \oplus \langle -2k \rangle$
and $m$ be an integer such that $md \leqslant 6$,  or $V = H^3 \oplus I_{16} \oplus \langle -2k \rangle$ and $m$ be an integer such that $md \leqslant  22$. Then there exists
a hermitian form $W$  over $E$ and a quadratic form $V'$ over $\QQ$ such that
$$
V \simeq {\TT}(W) \oplus V'~.
$$

Moreover, if $md = 6$, respectively $md = 22$, then $V' = \langle h \rangle$ with
$h = -2k \Delta_E$.

\end{prop}

\begin{proof} If $md < 6$, respectively $md < 22$, then this follows from Proposition \ref{Kum and K3[n] CM Hodge}. Assume 
that $md = 6$, respectively $md = 22$.

We have ${\rm det}(V) = 2k$ and $w(V) = (-1,-1) + (-1,-2k) = (-1,2k)$, and set $h = -2k \Delta_E$.
Since $H$ is an orthogonal factor of $V$, there exists
$x \in V$ such that $q(x,x) = h$; let $U$ be a quadratic form over $\bf Q$ such that $V \simeq U \oplus V'$. We have ${\rm det}(U) =
{\rm det}(V){\rm det}(V') = (2k)(-2k \Delta_E) = -\Delta_E$; this implies that  
$w(V) = w(U) + (-\Delta_E,-2k \Delta_E)$. The signature of $U$ is $(2,md-2)$,

If $p \in S_E$, then $\Delta_E = 1$ in
$\QQ_p^{\times}/\QQ_p^{\times 2}$, hence by the above computation we have $w(U) = (-1,2k) + (-1,-2k) = (-1,-1)$ at $p$; this 
 implies that $w(U) = w(H^3)$, respectively $w(H^{11})$. Therefore if $p \in S_E$, then $U \otimes_{\bf Q} {\bf Q}_p$
is isomorphic to $H^{3} \otimes _{\QQ} \QQ_p $, respectively $H^{11} \otimes _{\QQ} \QQ_p $.
By Theorem \ref{realization} there exists a hermitian form $W$ over $E$ such that $U \simeq \TT(W)$.

Conversely, assume that $V \simeq T(W) \oplus V'$ for some hermitian form $W$, and set $U = T(W)$. Since $md = 6$ or $md = 22$ and $d$ is even,
this implies that $m$ is odd; moreover, $d$ is not divisible by $4$. Therefore ${\rm det}(U) = -  \Delta_E$. We have ${\rm det}(V) = 2k$,
hence ${\rm det}(V') = - 2k \Delta_E$. Since ${\rm dim}(V') = 1$, this implies that $V' = \langle h \rangle$ with
$h = -2k \Delta_E$. 
\qed
\end{proof}

\medskip

The following two results follow directly from Theorem \ref{cm 1} combined with Proposition \ref{Kum and K3[n] CM} by taking $k=1$
and adding an orthogonal summand $\langle -2\rangle$ resp.\ $\langle -6\rangle$ to $V$ and $V'$:

\begin{coro}\label{OG6} Let $V = H^3 \oplus \langle -2, -2 \rangle$ and let $m$ be an integer with $m \geqslant 1$ such that $md \leqslant 6$.
Then there exists a hermitian form $W$ over $E$ such that $\TT(W)$ has signature $(2,md-2)$ and a quadratic form $V'$ over $\QQ$ such
that $$V \simeq {\TT}(W) \oplus V'.$$

\end{coro}

\begin{coro}\label{OG10} Let $V = H^3 \oplus I_{16} \oplus \langle -2,-6 \rangle$, 
and let $m$ be an integer with $m \geqslant 1$ such that $md \leqslant 22$.
Then there exists a hermitian form $W$ over $E$ such that $\TT(W)$ has signature $(2,md-2)$ and a quadratic form $V'$ over $\QQ$ such that
$$
V\, \simeq\, {\TT}(W) \oplus V'~.
$$

\end{coro}

Contrary to Proposition \ref{Kum and K3[n] CM}, Corollaries \ref{OG6}, \ref{OG10}
make no claim about the uniqueness of $V'$,
not even in the maximal dimensional case.
Instead, we have the
following:

\begin{prop} Let $V = H^3 \oplus \langle -2, -2 \rangle$ and let $m$ be an integer with $m \geqslant 1$ such that $md = 6$.
If $W$ is a hermitian form over $E$ such that $\TT(W)$ has signature $(2,md-2)$ and if  $V'$  is a quadratic form over $\QQ$ such that
$$
V\, \simeq\, {\TT}(W) \oplus V'~,
$$
then $V' \simeq \langle a,  a\Delta_E  \rangle$ with $a \in \QQ^{\times}$ such that if $p \in S_E$, then $(-1,-a) = 0$
at $p$.
\end{prop}

\noindent
{\bf Proof.} We have ${\rm det}(V) = -1$ and $w(V) = 0$. Set $U = {\TT}(W)$; then ${\rm det}(U) = -\Delta_E$, therefore ${\rm det}(V') = \Delta_E$.
Since ${\rm dim}(V') = 2$, this implies that $V' \simeq \langle a,  a\Delta_E  \rangle$ with $a \in \QQ^{\times}$.

We have $w(V) = w(U) + w(V')$, and $w(V) = 0$, hence $w(V') = w(U)$. 
If $p \in S_E$, then $w(U) = (-1,-1)$ locally at $p$, 
therefore
$w(V') = w(\langle a, a \Delta_E \rangle) = (-1,-1)$. 
On the other hand, if $p \in S_E$, then $\Delta_E = 1$ locally at $p$ by the comment after Notation \ref{notation};
hence we obtain
$w(V') = (a,a) = (-1,-1)$, and this equivalent to $(-1,-a) = 0$ for $p \in S_E$, as claimed
(cf.\ Example \ref{ex:calc} (1)).
\qed

\begin{prop} Let $V = H^3 \oplus I_{16} \oplus \langle -2,-6 \rangle$,  and let $m$ be an integer with $m \geqslant 1$ such that $md = 22$.
If $W$ is a hermitian form over $E$ such that $\TT(W)$ has signature $(2,md-2)$ and if  $V'$  is a quadratic form over $\QQ$ such that
$$
V\, \simeq\, {\TT}(W) \oplus V'~,
$$
then $V' \simeq \langle a,  a\Delta_E  \rangle$ with $a \in \QQ^{\times}$ such that if $p \in S_E$, then $(-3,-2a) = 0$
at $p$.
\end{prop}

\noindent
{\bf Proof.} We have ${\rm det}(V) = -3$ and $w(V) = (2,6)$. Set $U = {\TT}(W)$; then ${\rm det}(U) = -\Delta_E$ because 
$m$ and $d_0$ are both odd, therefore ${\rm det}(V') = 3\Delta_E$.
Since ${\rm dim}(V') = 2$, this implies that $V' \simeq \langle a,  3a\Delta_E  \rangle$ with $a \in \QQ^{\times}$.

 Let $p \in S_E$, then $\Delta_E = 1$, and $V' \simeq \langle a, 3a \rangle$. Since $V \simeq U \oplus V'$, we have
$w(V) = w(U) + w(V') + ({\rm det}(U),{\rm det}(V'))= (-1,-1) + w(V') + (-1,3) = w(V') + (-1,-3)$. We have $w(V) = (2,6)$, therefore
we obtain $w(V') = (2,6) + (-1,-3) = (-2,-6)$. Since $w(V') = (a,3a)$, this is equivalent to $(a,3a) = (-2,-6)$, and this in turn can
be reformulated as $(-3,-2a) = 0$ (cf.\ Example \ref{ex:calc} (2)).
\qed

\section{Totally real fields}
\label{s:RM}

We continue by investigating totally real fields,
as needed for the proof of Theorem A.

 If $E$ is  a totally real field, we denote by $\Sigma_E$ be the set of real embeddings of $E$; we have
$E \otimes_{\QQ}\RR =  \prod_ {\sigma \in \Sigma_E} E_{\sigma}$, with $E_{\sigma} = {\bf R}$ for all $\sigma \in \Sigma_E$.  If $W$ is a quadratic form over $E$, then
$W \otimes_{\QQ} \RR$ decomposes as an orthogonal sum $W \otimes_{\QQ} \RR =  \bigoplus_{\sigma \in \Sigma_E}  W_{\sigma}$;
each of the $W_{\sigma}$ is a quadratic form over $\RR$.

Again, the first results are quite general and do not require the K3 setting to which we will specialize soon.

\begin{theo}\label{Hodge}  Let $V$ be a quadratic form over $\QQ$ of signature $(r,s)$. Let $E$ be a totally real number field of
degree $d$, 
 and let $m \geqslant 1$ be an integer with $md \leqslant  {\rm dim}(V) - 2$. 
Let  $r',s' \geqslant 0$ be integers such that $r' \leqslant r$, $s'\leqslant s$, and $r'+s' = md$.

 Let
$W$ be a quadratic  form of dimension $d$ over $E$ such that the signature of $\TT(W)$ is $(r',s')$, and
suppose that one of the following holds: 
 \begin{itemize}
 \item[$\bullet$] 
 $md < {\rm dim}(V) - 2$,
 \item[$\bullet$] 
 $md =  {\rm dim}(V) - 2$, and moreover for all prime numbers $p$
such that $w(\TT(W)) \not = w(V)$ at $p$, we have
${\rm det}(\TT(W)) \not = - {\rm det}(V)$ in $\QQ_p^{\times}/\QQ_p^{\times 2}$.
\end{itemize}

Then there exists a quadratic form $V'$ over $\QQ$ such that
$$
V \simeq {\TT}(W) \oplus V'.
$$

\end{theo}

\begin{proof} This follows from Lemma \ref{useful}. 
\qed

\end{proof}

\begin{coro}\label{real} Let $V$ be a quadratic form over $\QQ$. Let $E$ be a totally real number field of
degree $d$, let $m$ be an integer with $m \geqslant 1$ such that $m d \leqslant  {\rm dim}(V) - 2$. Let $(r,s)$ be the
signature of $V$, and let $r',s' \geqslant 0$ be integers such that $r' \leqslant r$, $s'\leqslant s$, and $r'+s' = md$.
Suppose that $r' \leqslant m$.

 Then there exists a quadratic form $W$ over $E$ such that the signature of $\TT(W)$ is $(r',s')$ and a quadratic form $V'$ over $\QQ$ such
that $$V \simeq {\TT}(W) \oplus V'.$$

Moreover $W$ can be chosen in such a way that there is an embedding $\sigma: E \to \RR$ with $W_{\sigma}$  of signature $(r',m-r')$.
\end{coro}

\noindent
{\bf Proof.} Let $\sigma \in \Sigma_E$.
Let $a_1,\dots,a_{r'} \in E^{\times}$ be such that, for any $i=1,\hdots, r'$,
 $\sigma(a_i) > 0$ and $\tau(a_i) < 0$ for all $\tau \in \Sigma_E$ with $\tau \not = \sigma$, and
let $a_{r'+1},\dots,a_{m} \in E^{\times}$ be totally negative. Set $W = \langle a_1,\dots,a_m \rangle$. Then the
signature of $W_{\sigma}$  is $(r',m-r')$ and the signature of $\TT(W)$ is $(r',s')$.
If ${\rm dim}(V) - md = 2$, suppose moreover that for all prime numbers $p$
such that $w(\TT(W)) \not = w(V)$ at $p$, we have
${\rm det}(\TT(W)) \not = - {\rm det}(V)$ in $\QQ_p^{\times}/\QQ_p^{\times 2}$; this is possible by the weak approximation
theorem (see for instance \cite{C}, \S 6).
The previous theorem implies that there exists a quadratic form $W$ over $E$ such that the signature of $\TT(W)$ is $(r',s')$ and a quadratic form $V'$ over $\QQ$ such
that $V \simeq {\TT}(W) \oplus V'.$
\qed

\medskip

We continue with a result specifically geared towards the K3-setting (signature $(r=3, s=19)$) for the case
where $\TT(W)$ has codimension $1$.

This is also a turning point: the previously used methods do not suffice to handle this case. Instead, we
use some results of Kr\"uskemper \cite{K}; in the next sections, we adapt Kr\"uskemper's approach to
prove further results needed for our applications. However, for the next proposition, the results
already contained in \cite{K} suffice.

\begin{prop}\label{3,7} Let $V = V_{K3}$.  Let $V' = \langle h \rangle$ with $h > 0$.
Let $E$ be a totally real number field of degree $d$ with $d = 3$ or $7$, let $m=21/d$.
Then there exists a quadratic form $W$ over $E$ such that
$$
V \simeq {\TT}(W) \oplus V'~.
$$
Moreover $W$ can be chosen in such a way that there is an embedding $\sigma: E \to \RR$
with $W_{\sigma}$  of signature $(2,m-2)$.
\end{prop}

\noindent
{\bf Proof.} Let $X$ be a quadratic form over  $\QQ$ of signature $(1,d-1)$ such that ${\rm det}(X) = 1$ and $w(X) = 1$ at the
prime $2$ and at infinity, and $0$ elsewhere; this is possible by Theorem \ref{S}, (i). By a result of Kr\"uskemper \cite{K} Theorem A there exists a 1-dimensional quadratic form $Y$ over $E$ of  such
that $\TT(Y) \simeq X$; indeed, the hypotheses of this result are satisfied, since $w(X) = 0$ at all non-dyadic primes.  Set
$U_1 = X \oplus X$ and $W_1 = Y \oplus Y$; we have $\TT(W_1) \simeq U_1$. The quadratic form $U_1$ has signature
$(2,2d-2)$, trivial determinant and Hasse invariant.

Let $U_2$ be a negative definite quadratic form over  $\QQ$ of dimension $(m-2)d$, determinant $-h$ and such that
$w(U_2) = w(V)$; such a
quadratic form exists by Theorem \ref{S} (ii).
There exists a quadratic form $W_2$ over $E$ of  such that $\TT(W_2) \simeq U_2$;
this follows from \cite{K}, Theorem A if $m =3$, and from \cite{K}, Theorem 3 b) if $m = 7$.

Set $W = W_1 \oplus W_2$ and $U = U_1 \oplus U_2$.
Note that
$U$ is a quadratic form over $\QQ$ of signature $(2,19)$ with ${\rm det}(U) = -h$ and $w(U) = w(V)$.
We have $V' = \langle -{\rm det}(U) \rangle$;
we claim that $V \simeq U \oplus V'$. Indeed, these forms have the same dimension, determinant and signature.
We have $w(U \oplus V') = w(U) + w(V') + ({\rm det}(U),{\rm det}(V'))$. Since ${\rm det}(V') = - {\rm det}(U)$,
we have $({\rm det}(U),{\rm det}(V')) = 0$. Moreover, ${\rm dim}(V') = 1$, hence $w(V') = 0$. This implies
that $w(U \oplus V') = w(U) = w(V)$, and therefore $V \simeq U \oplus V'$, as claimed. Since $\TT(W) = U$,
we have $V \simeq {\TT}(W) \oplus V'.$

By construction, since $X=\TT(Y)$ has signature $(1,d-1)$,
there is an embedding $\sigma: E \to \RR$ with $Y_{\sigma}$ positive.
Since $U_2=\TT(W_2)$ is negative definite, it follows that $(Y\oplus Y \oplus W_2)_{\sigma}$ has signature $(2,m-2)$.
\qed

\medskip
For future reference, we collect the results relevant to the K3 setting:

\begin{coro}\label{for Theorem A} Let $E$ be a totally real number field of degree $d$ and let $m$ be an integer with $m \geqslant 3$
and $md \leqslant 21$. Let $V = V_{K3}$.
Then there exists a quadratic form $W$ over $E$ and a quadratic form $V'$ over $\QQ$ such that
$$
V \simeq {\TT}(W) \oplus V'~.
$$
Moreover $W$ can be chosen in such a way that there is an embedding $\sigma: E \to \RR$
with $W_{\sigma}$ of signature $(2,m-2)$
while all other embeddings $\tau: E \to \RR$ have $W_\tau$ negative-definite.
\end{coro}

\begin{proof}
If $md = 21$, this is Proposition \ref{3,7}. Suppose that $md \leqslant 20$; then the result follows from Corollary \ref{real}
with $r = 3$, $s = 19$, $r' = 2$ and $s' = md$.
\qed
\end{proof}

\medskip

The next  result will be used in the applications to HK manifolds (see Section \ref{s:HK}).

\begin{prop}\label{Kum and K3[n]} Let $k > 0$ be an integer, and let $V = H^3  \oplus \langle -2k \rangle$
and $m = 3$  or $V = H^3 \oplus I_{16} \oplus \langle -2k \rangle$ and $m = 11$.  Let
$E$ be a real quadratic field.

Let $h > 0$ be such that $-2kh \in \mathrm{N}_{E/{\bf Q}}(E^{\times})$, and set $V' = \langle h \rangle$. Then there exists
a quadratic form $W$ over $E$  such that
$$
V \simeq {\TT}(W) \oplus V'~.
$$
Moreover $W$ can be chosen in such a way that there is an embedding $\sigma: E \to \RR$
with $W_{\sigma}$ of signature $(2,m-2)$.

Conversely, if $V \simeq {\TT}(W) \oplus V'$ for some quadratic form $W$ over $E$ and $V'=\langle h \rangle$, then 
$-2kh \in \mathrm{N}_{E/{\bf Q}}(E^{\times})$.
\end{prop}

\noindent
{\bf Proof.} Set $U_1 = H^2$. There exists a quadratic form $W_1$ over $E$ such that $\TT(W_1) = U_1$ and
an embedding $\sigma: E \to \RR$
with $(W_1)_{\sigma}$ of signature $(2,2)$
(see \cite{GS}, 3.11).
Set $U' = H \oplus  \langle -2k \rangle$ or $H \oplus I_{16} \oplus \langle -2k \rangle$. The quadratic form $H$ represents all non-zero
rational numbers, hence there exists $x \in U'$ such that $q(x,x) = h$. Let $U_2$ be such that $U' \simeq U_2 \oplus \langle h \rangle$.
By \cite{K}, Proposition 6, there exists a form $W_2$ such that $\TT(W_2) = U_2$.
Set $W = W_1 \oplus W_2$.
and $U = U_1 \oplus U_2$.
By construction, we have $V \simeq {\TT}(W) \oplus V'$, and the signature of $W_{\sigma}$ is $(2,m-2)$.

Conversely, assume that  $V \simeq {\TT}(W) \oplus V'$. 
This implies that ${\rm det}({\TT}(W)) = {\rm det}(V) {\rm det}(V') = 2kh$.  Since  $E$ is a quadratic field, $- \Delta_E \in {\mathrm N}_{E/{\bf Q}}(E^{\times})$, 
and Lemma \ref{invariants bis} (ii)
 implies that $-{\rm det}({\TT}(W) )  \in {\mathrm N}_{E/{\bf Q}}(E^{\times})$,
hence $-2kh \in {\mathrm N}_{E/{\bf Q}}(E^{\times})$, as claimed.
\qed

\section{Witt groups}\label{Witt}

The aim of this section is to recall some notions and results concerning Witt rings of quadratic forms; 
see for instance \cite{L} or \cite{MH} for
details. 
Let $F$ be a field of characteristic $\not = 2$. We denote by ${\rm Witt}(F)$ the {\it Witt ring} of $F$, and by $I(F)$ the {\it fundamental ideal}
of ${\rm Witt}(F)$, i.e.\ the
ideal of the even dimensional quadratic forms. 

The previously defined invariants, 
such as dimension, determinant, signature and Hasse invariant, 
do not vanish on hyperbolic forms, hence they are not well-defined on ${\rm Witt}(F)$. 
Therefore, one has to modify these invariants. 
First of all, the dimension is only defined modulo $2$: 
in the sequel, the dimension of a Witt class will be understood as an element of $\ZZ/2\ZZ$. 
Instead of the determinant, we consider the discriminant, defined as follows.

If $V$ is a quadratic form of dimension $n$,  the {\it discriminant} of $V$ is defined  by setting ${\rm disc}(V) = (-1)^{n(n-1)/2}{\rm det}(V)$. Note that ${\rm disc}(H) = 1$, and ${\rm disc}$ induces a group homomorphism 
$${\rm disc} :
{\rm Witt}(F) \to F^{\times}/F^{\times 2}.
$$

Next, we modify the Hasse invariant, as in \cite[Chapter V, Propositions 3.19, 3.20]{L}, in order to obtain an invariant 
\begin{eqnarray}
\label{eq:Witt}
c : {\rm Witt}(F) \to \Br_2(F),
\end{eqnarray}
called the \emph{Witt invariant}. The relationship with the Hasse invariant can be expressed in a (rather complicated) formula, 
see \cite[Chapter V, Proposition 3.20]{L}.

If $E$ is a finite field extension of $F$, the transfer defined in Section \ref{transfer section}
induces a group homomorphism
$$
\TT :  {\rm Witt}(E) \to {\rm Witt}(F).
$$

Let ${\rm N}_{E/F} : E \to F$ be the norm map; 
it induces a homomorphism 
$$
{\rm N}_{E/F} : E^{\times}/E^{\times 2}  \to F^{\times}/F^{\times 2};
$$ 
let $\Lambda_{E/F}$ be the image of this homomorphism. Note that if $[E:F]$ is odd, then 
$\Lambda_{E/F} = F^{\times}/F^{\times 2}$; indeed,
if $a \in F^{\times}$ then ${\rm N}_{E/F}(a) = a^{[E:F]}$.

If $E$ is a totally real number field, we denote by $\Lambda^+_{E/F}$ the image of totally positive elements of $E$.

We say that a quadratic form $V$ is a {\it torsion form} if its class in ${\rm Witt}(E)$ is a torsion element of the group ${\rm Witt}(E)$. If
there is no ambiguity, we use the same notation for a form and its Witt class. 

\begin{example}
Let $E = \QQ$. 
The form $V = \langle 1,-5\rangle$ 
does not represent $0$ in $\rm{Witt}(\QQ)$, 
but it is a torsion form, of order $2$. 
Indeed, the form $V\oplus V$ has dimension 4, signature $(2,2)$, determinant $1$ (mod squares), 
while its Hasse invariant is $(-5,-5) = (-1,-5)$. 
This is non-zero at 2 and over the real numbers, 
and zero everywhere else locally, hence it is equal to the Hasse invariant of the hyperbolic form $H^2$. 
This implies that all the invariants of this form coincide with those of $H^2$, hence it is isomorphic to $H^2$.
\end{example}


\begin{theo}\label{torsion} Suppose that $E$ is a totally real number field, and let $U \in I(\QQ)$ be a torsion form. There exists a torsion form
$W \in I(E)$ such that $T(W) \simeq U$ if and only if ${\rm disc}(U) \in  \Lambda^+_{E/F}$.
\end{theo}

\noindent
{\bf Proof.} By a result of Kr\"uskemper \cite{K}, Lemma 7, such a $W \in I(E)$ exists if and only if 
${\rm N}_{E/{\bf Q}}(a) = {\rm disc}(U)$ in $\QQ^{\times}/{\QQ}^{\times 2}$ for some
$a \in E^{\times}$ that is a sum of squares in $E$.
Since $E$ is totally real,
an element  $a \in E^{\times}$ is a sum of
squares if and only if it is totally positive (see for instance \cite{L}, Chapter VII, Theorem 1.12). 
\qed

\medskip 
Suppose that $E$ is a totally real number field, and let $V$ be a quadratic form over $E$; let $\sigma \in \Sigma_E$. 
The {\it index} of $V$ at $\sigma$, denoted by ${\rm ind}_{\sigma}(V)$, is by definition $r_{\sigma} - s_{\sigma}$, where
$(r_{\sigma},s_{\sigma})$ is the signature of $V_{\sigma}$. This induces a group homomorphism  ${\rm ind}_{\sigma} : {\rm Witt}(E) \to {\ZZ}$.

\begin{theo}\label{Pfister}  Suppose that $E$ is a totally real number field, and let $V \in {\rm Witt}(E)$. Then $V$ is 
a torsion element of the group ${\rm Witt}(E)$ if and only if ${\rm ind}_{\sigma}(V) = 0$ for all $\sigma \in \Sigma_E$. 

\end{theo}

\noindent
{\bf Proof.} This is a consequence of Pfister's local-global principle, see for instance \cite{MH}, Corollary 3.12. 
\qed

\section{Characterization}\label{characterization section}

Before turning to the setting of K3 surfaces,
we discuss a general characterization of  quadratic forms arising by transfer
from a suitable number field.
We will return to this general set-up in  Section \ref{ss:general}.

\medskip

Let $U$
be a quadratic form over $\QQ$ of dimension $r$ and signature $(2,r-2)$. Let $E$ be a number field of degree $d$, and let $m$ be
an integer such that $r = md$;  assume that either $E$ is CM and $m \geqslant 1$ or $E$ is totally real and $m \geqslant 3$.

\begin{question}
Does there exist a hermitian (if $E$ is CM) or quadratic (if $E$ is totally real) form $W$ over $E$ such that
\begin{itemize}
\item[$\bullet$] 
$\TT(W) \simeq U$;
\item[$\bullet$] 
If $E$ is totally real,
is there  a $\sigma\in\Sigma_E$ such that the signature of $W_{\sigma}$ is $(2,m-2)$?
\end{itemize}
\end{question}

In the CM case, the characterization of the quadratic forms $U$ having this property follows from
the results of  Section \ref{transfer section}.

\begin{theo}\label{CM characterization} 
In the above setting, suppose that $E$ is a CM field.
There exists a hermitian form $W$ over $E$ such that $U \simeq \TT(W)$ if and only if the following two conditions hold:

\begin{enumerate}
\item
${\rm dim}_{\QQ}(U) = m[E:\QQ]$.

\item
${\rm disc}(U) = \Delta_E^m$;

\item
 If $p \in S_E$, then $U \otimes_{\QQ} {\QQ}_p$ is isomorphic to an orthogonal sum of hyperbolic planes. 

\item
The signature of $U$ is of the form $(2a,2b)$ for some integers $a,b \geqslant 0$.
\end{enumerate}
\end{theo}

\begin{proof} 
This follows from  Theorem \ref{realization}. Indeed, ${\rm dim}(U) = 2d_0m$, hence ${\rm disc}(U) = (-1)^{d_0m}{\rm det}(U)$; this
implies that condition (ii) above and condition (ii) of Theorem \ref{realization} are equivalent. \qed
\end{proof}

\medskip

These conditions are rather restrictive; as we will see, the case of totally real fields is quite different. 
Before we proceed, we need another auxiliary result.

\begin{lemma}
\label{lem:represent}
Let $E$ be a number field, let $W \in {\rm Witt}(E)$ and let $m\geq 3$ be an integer. 
Suppose that
$\dim(W)\equiv m \mod 2$, and that the indices of $W$ have absolute value $|{\rm ind}_\sigma(W)|\leq m$ at all the real places $\sigma$ of $E$. 
Then there exists a quadratic form $V$ over $E$ of dimension $m$ representing the Witt class $W$.
\end{lemma}

\begin{proof}  
Let $c \in \Br_2 E$ be the Witt invariant of $W$ from Section \ref{Witt}, 
and let $w \in \Br_2 E$ correspond to $c$ as in \eqref{eq:Witt}.
Since $m \geq 3$ and the indices of $W$ are $\leq m$ at all the real places of E, 
there exists a quadratic form $V$ of dimension $m$ over $E$ of Hasse invariant $w$, 
with the same indices and discriminant as $W$; this follows from \cite[72:1 and 63:23]{OM} 
(these results generalize Theorem  \ref{S} to number fields).
The form $V$ represents the Witt class $W$.
\qed
\end{proof}

Concentrating on totally real fields,
we start with the
case where $d$ is odd, and apply a method of Kr\"uskemper, \cite{K}. 
 
 \begin{theo}\label{new} 
 In the above setting, suppose that $E$ is a totally real field of odd degree $d$ 
 and let $m \geqslant 3$ be an integer such that $r = dm$. 
Then there exists a quadratic form $W$ over $E$ such that 
$$
U \simeq {\TT}(W).
$$
Moreover $W$ can be chosen in such a way that there is an embedding $\sigma: E \to \RR$
with $W_{\sigma}$  of signature $(2,m-2)$.
\end{theo}

\begin{proof} 
Let $\sigma \in \Sigma_E$ and let  $\alpha_1,\dots,\alpha_m \in E^{\times}$ such that $\sigma(\alpha_1) > 0$, 
$\sigma(\alpha_2) > 0$, that $\tau(\alpha_1) < 0$ and $\tau(\alpha_2) < 0$ for all $\tau \in \Sigma_E$
with $\tau \not = \sigma$, and that $\tau(\alpha_i) < 0$ for  all $i = 3,\dots,m$ and   all $\tau \in \Sigma_E$.
Set $W' = \langle \alpha_1,\dots,\alpha_m \rangle$. Note that the signature of $\TT(W')$
is equal to the signature of $U$. This implies that the Witt class of $\TT(W') - U$ is a torsion element of ${\rm Witt}(\QQ)$
(see Theorem \ref{Pfister}).

We follow the proof of Kr\"uskemper in \cite{K}, Proof of Theorem 3 (b), page 115. Let us consider the Witt class $X = \TT(W') - U$ in 
${\rm Witt}(\QQ)$, and note that $X \in I(\QQ)$. 
By \cite{K}, Corollary of Lemma 7, page 114, there exists a torsion class
$Y \in I(E)$ such that $\TT(Y) = X$. 
(Note that  applying this corollary requires the degree $d$ to be odd.)
Since $Y$ is torsion, its indices are $0$ at all the real places of $E$ (see Theorem \ref{Pfister}). 
This implies that the indices of $Y-W'$ are $\leq m$ at all the real places of $E$. 
Note that the dimension of $Y$ is even, since
$Y \in I(E)$, and that  $\dim(W')= m$, hence the dimension of the Witt class
$Y-W'$ is congruent to $m$ modulo $2$. Since moreover $m \geq 3$, we can apply
Lemma \ref{lem:represent} and conclude that there exists a quadratic form $W$ of dimension $m$ over $E$ 
representing the Witt class $Y-W'$. Then
$\TT(W) = U$ in ${\rm Witt}(\QQ)$, and since ${\rm dim}(\TT(W)) = {\rm dim}(U)$, we have
$\TT(W) \simeq  U$.
\qed
\end{proof}

\medskip

The analogous statement does not hold in general when $d$ is even, as shown by the following lemma:

\begin{lemma}\label{condition for even degree} 
 In the above setting, suppose that $E$ is a totally real field of even degree $d$ 
 and let $m \geqslant 3$ be an integer such that $r = dm$.  
Suppose that there exists a quadratic form $W$ over $E$ such that 
$$
U  \simeq {\TT}(W).
$$
Moreover, assume that there is an embedding $\sigma: E \to \RR$
with $W_{\sigma}$  of signature $(2,m-2)$.

Then we have $${\rm det}(U) \in \Lambda^+_{E/{\QQ}} \Delta_E^m.$$
\end{lemma}

\noindent
{\bf Proof.} Let $\alpha_1,\dots,\alpha_m \in E^{\times}$ be such that $W = \langle \alpha_1,\dots,\alpha_m \rangle$; we may
assume that $\sigma(\alpha_1) > 0$, $\sigma(\alpha_2) > 0$ and $\tau(\alpha_1) < 0$ and $\tau(\alpha_2) < 0$ for all $\tau \in \Sigma_E$
with $\tau \not = \sigma$, and that $\alpha_3,\dots,\alpha_m$ are totally negative. We have ${\rm det}(W) = \alpha_1 \dots \alpha_m$. 
Note that this element is totally positive if $m$ is even, and totally negative if $m$ is odd. Set $\alpha = {\rm det}(W)$ if $m$ is even,
and $\alpha = - {\rm det}(W)$ if $m$ is odd; note that since $d$ is even, ${\mathrm N}_{E/\QQ}(x) = {\mathrm N}_{E/\QQ}(-x)$ for all $x \in E^{\times}$. 
We have ${\rm det}(U) = {\mathrm N}_{E/\QQ}({\rm det}(W)) \Delta_E^m = {\mathrm N}_{E/\QQ}(\alpha) \Delta_E^m$, and $\alpha$ is totally positive. This
completes the proof of the lemma. 
\qed

\medskip
The following result shows that the condition of Lemma \ref{condition for even degree} is sufficient:

\begin{theo}\label{even degree} 
 In the above setting, suppose that $E$ is a totally real field of even degree $d$ 
 and let $m \geqslant 3$ be an integer such that $r = dm$. 
There exists a quadratic form $W$ over $E$ such that
$$
U \simeq {\TT}(W)
$$
if and only if ${\rm det}(U) \in \Lambda^+_{E/{\QQ}} \Delta_E^m,$ i.e.\ 
if and only if there  exists a totally positive element
$\alpha \in E^{\times}$ such that
\begin{eqnarray}
\label{eq:cond_N}
{\rm det}(U) = {\mathrm N}_{E/\QQ}(\alpha) \Delta_E^m \ {\rm in} \ {\QQ}^{\times}/{\QQ}^{2 \times}.
\end{eqnarray}
Moreover $W$ can be chosen in such a way that there is an embedding $\sigma: E \to \RR$
with $W_{\sigma}$  of signature $(2,m-2)$.
\end{theo}

\begin{proof}
Let $\sigma \in \Sigma_E$ and let  $\alpha_1, \alpha_2 \in E^{\times}$ such that $\sigma(\alpha_1) > 0$, 
$\sigma(\alpha_2) > 0$, that $\tau(\alpha_1) < 0$ and $\tau(\alpha_2) < 0$ for all $\tau \in \Sigma_E$
with $\tau \not = \sigma$, and let $\alpha_3,\dots,\alpha_m \in E^{\times}$ be totally negative. 
Set $W' = \langle \alpha_1,\dots,\alpha_m \rangle$. 
Let us consider the Witt class $X = \TT(W') - U$ in 
${\rm Witt}(\QQ)$, and note that $X \in I(\QQ)$ and that it is torsion (cf. Theorem \ref{Pfister}).

The determinant of $\TT(W')$ is  equal to ${\mathrm N}_{E/\QQ}({\rm det}(W')) \Delta_E^m$. We have ${\rm det}(W') = \alpha_1 \cdots \alpha_m$; as
in the proof of Lemma \ref{condition for even degree}, we see that ${\rm det}(W')$ is totally positive if $m$ is even, and totally negative
if $m$ is odd. Set $\beta = {\rm det}(W')$ if $m$ is even, and $\beta =  - {\rm det}(W')$ if $m$ is odd, and note that the discriminant
of $X$ is ${\mathrm N}_{E/\QQ}(\alpha \beta) \Delta_E^{2m}$.

The element $\alpha \beta$ is totally positive, 
and $\Delta_E^{2m}$ is a square; hence
the discriminant of $X$ belongs to $\Lambda^+{E/Q}$.
Therefore,
by Theorem \ref{torsion} 
there exists a torsion form 
$Y \in I(E)$ such that $\TT(Y) = X$. Let $W$ be a quadratic form of dimension $m$ over $E$ representing the Witt class $Y - W'$;
this exists by 
Lemma \ref{lem:represent} as in the proof of Theorem \ref{new}
because $m>2$.
Then $\TT(W) = U$ in ${\rm Witt}(\QQ)$, and since ${\rm dim}(\TT(W)) = {\rm dim}(U)$, we have
$\TT(W) \simeq  U$.
\qed
\end{proof}

\medskip
The following lemma is due to Kr\"uskemper

\begin{lemma}\label{+} Let $E$ be a totally real field. If $E/{\bf Q}$ is a Galois extension and $a\in \Lambda_{E/{\QQ}}$ satisfies $a>0$, then $a\in \Lambda^+_{E/{\QQ}}$.
\end{lemma}

\noindent
{\bf Proof.} This follows from \cite{K}, Proposition 7, (b), and the fact that every totally positive element of $E$ is
a sum of squares.
\qed

\begin{example}\label{Bert's example} 
\label{ex:Bert}
Let $E$ be a real quadratic field such that $E = \QQ(\sqrt d)$ with $d$ odd and square free,
and let $m \geqslant 3$ be an odd integer. Let $U$ be a quadratic form over $\QQ$ of dimension $2m$ and of determinant $1$. 
Then there exists an $m$-dimensional quadratic form $W$ over $E$ such that $\TT(W) \simeq U$ if and only if $d$ is a sum of
two squares, i.e.\ there exist $a,b \in {\bf Z}$ such that $d = a^2 + b^2$.

Indeed, $d$ is a sum of two squares 
$\iff$ $\Delta_E \in {\rm N}_{E/{\QQ}}(E)$ 
$\iff$ ${\rm det}(U) \in  \Lambda_{E/{\QQ}} \Delta_E \iff  \det(U)\in \Lambda^+_{E/{\QQ}} \Delta_E^m$; 
 here the first equivalence was covered in Example \ref{ex:sum2},
the second equivalence is obvious,
and the last equivalence follows from Lemma \ref{+}. By Theorem \ref{even degree}, we have
${\rm det}(U) \in  \Lambda^+_{E/{\QQ}} \Delta_E^m$ $\iff$ there exists an $m$-dimensional quadratic form $W$ over $E$ such that $\TT(W) \simeq U$.
\end{example}

\section{Hodge structures of K3 type and geometrical realizations}
\label{ss:general}

\subsection{Hodge structures of K3 type}\label{ss: K3type}
We recall that a polarized (integral) Hodge structure $(U_\ZZ,q)$, where $U_\ZZ$ is a free $\ZZ$ module of rank $r$,
is of K3 type if it has
weight two and the Hodge decomposition
$$
U_\CC\,=\,U_\ZZ\otimes_\ZZ\CC\,=\,U^{2,0}\,\oplus\,U^{1,1}\,\oplus U^{0,2}
$$
is such that $\dim U^{2,0}=1$.
Then $U^{2,0}=\CC\omega$ for some $\omega\in U_\CC$, and $U^{0,2}=\CC\overline{\omega}$.

Such a Hodge structure is simple, that is, does not have non-trivial Hodge substructures, if and only if
$U_\ZZ\cap (U^{2,0})^\perp=0$.

The polarization $q$ of the Hodge structure $U_\ZZ$ is a quadratic form on $U_\ZZ$ such that
$U^{2,0}$ and $U^{0,2}$ are isotropic subspaces,
$q$ is positive definite on the real two-dimensional subspace
$$
U_2:=(U_\ZZ\otimes_{\ZZ}\RR)\cap (U^{2,0}\oplus U^{0,2})
\;\;\; \text{ and } \;\;\;
U^{1,1}=(U^{2,0}_X\oplus U^{0,2}_X)^\perp.
$$
The polarization is negative definite
on $U_0:=(U\otimes_{\ZZ}\RR)\cap U^{1,1}_X$ and thus has signature $(2,r-2)$ on $U_\RR=U_\ZZ\otimes_\ZZ\RR$.

A simple K3 type Hodge structure is a member of an $r-2$-dimensional family
of K3 type Hodge structures, the general one again being simple, which is parametrized by the choice
of an $\omega\in U\otimes_\ZZ\CC$ with $q(\omega,\omega)=0$ and $q(\omega,\bar{\omega})>0$; the Hodge structure
$U_\omega$ determined by $\omega$ has $U^{2,0}_\omega=\CC\omega$.

Any rational polarized Hodge structure $(U,q)$ of K3 type can be realized geometrically.
In fact, it is a Hodge substructure of
$H^2(KS(U),\QQ)$ where $KS(U)$ is the Kuga Satake abelian variety $KS(U)$ of $U$, which is defined in terms of the Clifford algebra of $(U,q)$ (cf.\ \cite{vG00}).
Combined with Lefschetz theorems, it follows that these Hodge structures also appear in the second cohomology group of certain projective surfaces. 
This is quite exceptional, the results on variations of Hodge structures 
(see \cite{CGG}) imply that a general Hodge structure of weight $2$ with 
$h^{2,0}>1$ or of weight greater than $2$ cannot be a Hodge substructure of a smooth projective variety.

\subsection{Endomorphism algebras of K3 type Hodge structures}
The endomorphism algebra $A_U$ of a Hodge structure $U_\ZZ$ of K3 type is defined as, with $U=U_\ZZ\otimes_\ZZ\QQ$,
$$
A_U\, =\, {\rm End}_{\rm Hdg}(U)\,=\,\{f\in {\rm End}(U):\;f(U^{p,q})\subset U^{p,q}\}~.
$$
Zarhin \cite{Z}
showed that $A_U$ is either a CM field or a totally real field
by considering the action of an $f\in A_U$ on $U^{2,0}\simeq\CC$ and the polarization on $U$.
The adjoint of $a\in A_U$ for the polarization is the complex conjugate of $a$:
$$
q(ax,y) \,=\,q(x,\bar{a}y)~.
$$
This adjoint property is equivalent to the existence of an isometry of quadratic forms $U\simeq \TT(W)$ for some quadratic (or hermitian) form $W$ over the RM (or CM) field $A_U$, see Lemma \ref{transfer lemma}.
In the RM case, the form $W$ must be such that $W_\sigma$ has signature $(2,m-2)$ where $m=\dim W$ and $m\geq 3$
for one embedding and thus $W_\tau$ is negative definite for all other embeddings.
For totally real fields, there is the following restriction, but  for CM fields there is none.

\begin{lemma}[{\cite[Lem.\ 3.2]{G}}]
\label{lem:end}
If $E=A_U$ is a totally real field, then $m=\dim_ET_{X,\QQ}\geq 3$.
\end{lemma}

The results in Section \ref{characterization section} characterize the fields $E$, totally real or CM,
such that $(U,q)=\TT(W)$.
The results in Sections \ref{K3 section}, \ref{HK+RM} and \ref{HK+CM} apply these
criteria to specific quadratic forms that arise from Hodge substructures $T\subset H^2(X,\ZZ)$ of K3 type, where
$X$ is a hyperk\"ahler manifold.

\subsection{$K3$ surfaces}\label{preliminary section}
\label{s:K3}

We recall the basic facts on K3 surfaces needed for the application of the results obtained in the previous sections.
The reader may consult \cite{H} for in depth details.

A (complex, algebraic) K3 surface is a smooth projective surface $X$ with $\dim H^1(X,\mathcal O_X)=0$ and
trivial canonical bundle. The second cohomology group of $X$ is a free $\ZZ$-module of rank $22$ and the
intersection form on this group is a unimodular even bilinear form of signature $(3,19)$, so as a lattice
$$
H^2(X,\ZZ)\simeq \Lambda_{3,19}=H^3\oplus E_8^2
$$
and $V_{K3}=H^2(X,\ZZ)\otimes_{\ZZ}\QQ$.

The Hodge decomposition is a direct sum
$$
H^2(X,\ZZ)\otimes_{\ZZ}\CC\,=\,H^{2,0}(X)\oplus H^{1,1}(X)\oplus H^{0,2}(X),\qquad
H^{0,2}(X)\,=\,\overline{H^{2,0}(X)}~,
$$
and $H^{2,0}(X)=H^0(X,\Omega^2_X)=\CC\omega$, where $\omega\in H^2(X,\CC)$
is (the class of) a(ny) holomorphic 2-form which is nowhere zero.
The Picard group $\Pic(X)$ of $X$,
which parametrizes line bundles up to isomorphism,
is identified with a primitive sublattice of rank $\rho\geq 1$
of $H^2(X,\ZZ)$ of signature $(1,\rho-1)$ and $\Pic(X)\subset H^{1,1}(X)$.
The transcendental lattice $T_X$ of $X$ is defined as $\Pic(X)^\perp$,
it has rank $22-\rho$ and signature $(2,20-\rho)$.

The Hodge structure on $H^2(X,\ZZ)$ induces a non-trivial Hodge structure on the transcendental lattice
$$
T_X\otimes_{\ZZ}\CC= \oplus T_X^{p,q}, \;\;\; \text{ with } \;\;\;
T_X^{p,q}:=(T_X\otimes_{\ZZ}\CC)\cap H^{p,q}(X).
$$
The intersection form $q$ on $H^2(X,\ZZ)$ defines a polarization of the Hodge structure $T_X$,
and $(T_X,q)$ is a simple Hodge structure of K3 type.

\begin{example}
\label{ex:Bert2}
Let $X$ be a K3 surface arising as double cover of $\mathbb P^2$ branched along six lines 
(intersecting  in triple points at worst).
Then $\rho(X)\geq 16$, so $\rk (T_X)\leq 6$
and for $X$ to have RM by some totally real field $E$,
Lemma \ref{lem:end} forces equality $\rk (T_X)= 6$
and $E=A_X=\QQ(\sqrt d)$ to be real quadratic.
In this case, $\Pic(X)$ is a finite index overlattice
of $\langle 2\rangle\oplus A_1^{15}$,
the sublattice generated by the pull-back of a  hyperplane section of $\mathbb P^2$ and
by the 15 exceptional curves above the intersection points of the lines. 
In particular, $\det(\Pic(X)\otimes\QQ) =-1$
and $\det(T_{X,\QQ})=1$.

It then follows from Example \ref{Bert's example} that $d$ is a sum of two squares.
In \cite[Ex.\ 3.4]{G}, it was shown that if $d$ is odd and a sum of two squares then there exist such $X$ with
$A_X=\QQ(\sqrt d)$, in \cite{EJ 14} it was shown that $d$ can be any sum of two squares,
but is noteworthy that Example \ref{Bert's example} excludes any other real field.
Explicit examples are given in \cite{EJ 14} for $d=2$ and in \cite{EJ 23} for $d=2,5$.
\end{example}

\subsection{Hyperk\"ahler manifolds}\label{HK}
\label{s:HK}

A hyperk\"ahler (HK) manifold,
also called irreducible holomorphic symplectic manifold (IHSM),
is a simply connected compact complex K\"ahler manifold $X$ with trivial canonical bundle,
see \cite{Huy}, \cite{H-HK} for comprehensive overviews covering all the prerequisites for this paper.

A 2-dimensional HK manifold is a K3 surface. Currently we know of four higher dimensional
families of such manifolds, listed in
Table \ref{table} (following \cite[p.\ 78]{Rap})
\begin{itemize}
\item
the deformations of the generalized Kummer varieties $K_n(T)$,
where $T$ is a complex 2-dimensional torus,
\item
the deformations of the Hilbert schemes $S^{[n]}$, where $S$
is a K3 surface,
\item
the families of HK manifolds OG6, OG10 discovered by O'Grady.
\end{itemize}
The second integral cohomology group $H^2(X,\ZZ)$ of a HK manifold shares many of the properties of
the one of a K3 surface. It is a free $\ZZ$-module of rank $r=b_2(X)$ with an even quadratic form $q$
(not unimodular in general), the
Beauville-Bogomolov-Fujiki-form (or BBF-form), of signature $(3,r-3)$. 
A HK manifold $X$ is projective if and only if $\Pic(X)$ contains an element $D$ with $q(D,D)>0$.
The Hodge decomposition for a projective HK manifold has
$$
H^{2,0}(X)\,=\, \CC\omega \;\;\; \text{ and } \;\;\;
\Pic(X)=\omega^\perp\cap H^2(X,\ZZ)~,
$$
where $\omega\in H^2(X,\CC)$ is (the class of) a nowhere zero holomorphic two form.
The BBF-form defines a polarization $q$ on the Hodge structure $T_X:=\Pic(X)^\perp$
and $(T_X,q)$ is a simple Hodge structure of K3 type. We write $L_{r,n}$ for the lattice $(H^2(X,\ZZ),q)$ (which
only depends on $r=b_2(X)$ and $n=\dim X$):
$$
L_{r,n}\,\simeq\,(H^2(X,\ZZ),q),\quad V_{r,n}\,:=\,L_{r,n}\otimes_\ZZ\QQ\,\simeq\, (H^2(X,\QQ),q)~.
$$

\begin{table}[ht!]
  \begin{center}
    \label{HKs}
    \begin{tabular}{|c|c|c|c|c|}
    \hline
      $X$ & $\dim X$ &  $b_2(X)$ & $(H^2(X,\ZZ),q)$ & $(H^2(X,\QQ),q)^{\phantom{X}}_{\phantom{X}}$\\
      \hline
      $K_n(T)$ &$2n\;(n\geq 2)$ & 7 & $H^3\oplus\langle -2n-2\rangle$ & $H^3\oplus\langle -2n-2\rangle$\\
      OG6 & 6 & 8 & $H^3\oplus\langle -2, -2 \rangle$ & $H^3\oplus\langle -1, -1 \rangle$\\
      $S$ &2&22& $H^3\oplus E_8^2$ & $H^3 \oplus I_{16}$\\
      $S^{[n]}$ & $2n\;(n\geq 2)$ & 23 & $H^3\oplus E_8^2\oplus  \langle -2n+2\rangle $ & $H^3 \oplus I_{16} \oplus \langle -2n + 2 \rangle$\\
      OG10& 10& 24&$H^3\oplus E_8^2\oplus A_2$ & $H^3 \oplus I_{16} \oplus \langle -2,-6 \rangle$\\
\hline

    \end{tabular}
  \end{center}
  \caption{The known families of HK manifolds}
  \label{table}
\end{table}

\subsection{Surjectivity of the period map for HK manifolds}
The main result from complex geometry that we need to establish the existence of HK manifolds with RM or CM
is the surjectivity of the period map. This is the fact (\cite[Theorem 8.1]{Huy}), 
that if $L_{r,n}=(V_\ZZ,q_V)$ is a lattice which
is isometric to $(H^2(X,\ZZ),q)$ for a HK manifold $X$ and given an
$\omega\in V_\ZZ\otimes_{\ZZ}\CC$ with $q_V(\omega,\omega)=0$ and
$q_V(\omega, \overline{\omega})>0$, then there exists a HK manifold $X_\omega$ 
with an isometry $(H^2(X_\omega,\ZZ),q)\simeq (V_\ZZ,q)$ that maps $H^{2,0}(X_\omega)$ to $V^{2,0}:=\CC\omega$.
In standard terminology, this concerns the image of $\omega$ in the \emph{period space}
$$
{\mathcal P}_{r,n}\,:=\,\{[\omega]\in \mathbb P(V_\ZZ\otimes_\ZZ\CC):\;q_V(\omega,\omega)\,=\,0,\quad
q_V(\omega,\overline{\omega})>0\}.
$$
Moreover, if there  exists an $h\in V_\ZZ\cap\omega^\perp$ with $q_V(h,h)>0$,
then $X_\omega$ is  a projective HK manifold. From the Lefschetz $(1,1)$-theorem it follows that
\begin{eqnarray}
\label{eq:Pic(X)}
\Pic(X_\omega)\,=\,\{x\in H^2(X_\omega,\ZZ):\,q(x, \omega)\,=\,0\}.
\end{eqnarray}
From this one finds the transcendental lattice of $X_\omega$ as $T_{X_\omega}=\Pic(X_\omega)^\perp$.

For K3 surfaces
there is the Torelli theorem which asserts that $X_\omega$ is uniquely determined by $\omega$ up to isomorphism.
This does not hold in general for HK manifolds, see \cite{H-HK}.
%

\subsection{From lattices to quadratic forms and back}
\label{ss:back}

The results about the surjectivity of the period map in the previous section (as well as the Torelli theorems)
are phrased in terms of lattices as opposed to quadratic forms, the language used throughout the rest of this paper.

It is obvious that one can go from lattices to quadratic forms by tensoring with $\QQ$;
the same goes for integral and rational Hodge structures, the latter being our key target objects 
in view of Theorems A and B.

To go back from quadratic forms to lattices, 
one just has to fix an integral structure $L$ on the reference quadratic form $V$ of dimension $r$,
i.e.\ a free $\ZZ$-module $L\subset V$  of rank $r$
on which the quadratic form takes integer values.
In most cases this will be a primitive sublattice of $L_{r,n}\subset V_{r,n}$.

\section{HK manifolds with RM or CM}

\subsection{The condition (C)}
\label{ss:(C)}

Given one of the lattices $L$ from the fourth column of Table 1, to establish the existence of an HK manifold $X$ with an isometry $H^2(X,\ZZ)\simeq L$
which has real or complex multiplication with a field $E$, one has to choose a quadratic form $U\subset H^2(X,\QQ)$
for which the following Condition (C) is satisfied.

\begin{defn}
\label{def:C}
Let $E$ be a totally real or CM number field of degree $d$.
Let $U = (U,q)$ be a quadratic form over $\QQ$ of signature $(2,r-2)$.

We say that {\it condition} (C) {\it holds for $E$ and $U$} if $U$ has an $E$-module structure such that
$U\simeq \TT(W)$
for a quadratic or hermitian form $W$ over $E$ and, in case $E$ is totally real,
the eigenspace decomposition
\[
U_{\RR} =  \oplus_{\sigma: E\hookrightarrow\RR} U_\sigma
\]
contains an eigenspace $U_{\sigma_0}$ of signature $(2,m')$ with $m'>0$,
for the other embeddings $\tau$ the eigenspace $U_\tau$ is then negative definite;
moreover, $\dim_E U=2+m'\geq 3$
(as required in Lemma \ref{lem:end}).
\end{defn}

\subsection{Geometric realizations}

Let $(U,q)$ a quadratic form over $\QQ$  which is a direct summand of the quadratic form
$V_{r,n}\simeq (H^2(X,\QQ),q)$ as in Table 1; notice that this quadratic form only depends on $r=b_2(X)$
and $n=\dim X$.
Let $E$ be a field such that $U$ admits an $E$-module structure satisfying condition {\rm (C)}.
Then the choice of a suitable, general, $\omega$ in a suitable eigenspace of the $E$-action on $U_\CC$ determines
a Hodge structure on $V_{r,n}$ and thus on $L_{r,n}\subset V_{r,n}$.
Assuming that we are in one of the four known deformation types from Table \ref{table},
this determines a finite number of HK manifolds $X$
by the surjectivity of the period map  (and $X$ is unique if $r=22$),
together
with a Hodge isometry $U\simeq T_{X,\QQ}$ and $E\subset A_X$ (cf.\ \cite{H-HK});
equality can often be established with dimension arguments.

The results obtained in the previous sections on quadratic forms allow us to find examples of such $E$ and $U$.
Indeed, for fixed $E$-action on $V_{r,n}=L_{r,n}\otimes_\ZZ\QQ$, 
the periods $[\omega]\in {\mathcal P}_{r,n}$ that define HK manifolds $X$ with RM resp.\ CM by $E$ are those in a
complex submanifold of the period space.
 (to be made explicit in the proof of Thm 12.2).
 We refer to the dimension of the submanifold as the dimension of the family of HK manifolds of dimension $2n$ and $b_2(X)=r$.

With this in mind,
the following result provides the key geometric ingredient to prove Theorems A and B.

\begin{theo}
\label{thm:Torelli}
Let $V_{r,n}=L_{r,n}\otimes_\ZZ\QQ$ where $L_{r,n}$ is a lattice from Table 1.
Let $E$ be a totally real or CM field. Let  $U\subset V_{r,n}$ be a  quadratic form of signature $(2,r-2)$
which admits an $E$-module structure satisfying condition {\rm (C)} and let $m=\dim_EU$.

If $E$ is a totally real field and $m\geq 3$,
there is a family of complex projective HK manifolds $X$ with $b_2(X)=r$
such that a very general member $X$ has the properties
that $T_{X,\QQ} \simeq U$ as quadratic forms and $E =  {\rm End}_{\rm Hdg}(T_{X,\QQ})$.
The dimension of this family is  $m-2$.

If $E$ is a CM field and $m> 1$, there is a family of complex projective HK manifolds $X$ with $b_2(X)=r$
such that a very general member $X$ has the properties
that $T_{X,\QQ} \simeq U$ as quadratic forms
and $E =  {\rm End}_{\rm Hdg}(T_{X,\QQ})$.
The dimension of this family is $m-1$.

If $m=1$ (and $E$ is a CM field),
then there exists an  HK manifold $X$ with the above properties.

In all these cases, for each very general and also for the single HK manifold $X$,
we have $\Pic(X) \cong U^\perp\cap L_{r,n}$.
\end{theo}

\noindent
{\bf Proof.}
In case $E$ is totally real, any
$\omega\in \mathbb P(U_{\sigma_0}\otimes_\RR\CC)$ with $q(\omega,\omega)=0$ and $q(\omega,\bar{\omega})>0$
defines a Hodge structure on $U\subset V_{r,n}$ with $U^{2,0}=\CC\omega$ which induces one on $V_{r,n}$
and thus on $L_{r,n}$.
By the surjectivity of the period map,
$L_{r,n}$ is Hodge isometric to $H^2(X,\ZZ)$ for a HK manifold $X$  (cf.\ \cite{GS}, Proposition 3.7 for the K3 case).
Deformations of $\omega\in \mathbb P(U_{\sigma_0}\otimes_\RR\CC)$ induce deformations of $X$.
For a very general $\omega$ the Hodge structure on $U$ is simple and
there is a Hodge isometry $U\simeq T_{X,\QQ}$ of K3 type Hodge structures.
Notice that $\omega$, a point of an open subset in a quadric in a projective space of dimension $m-1$,
depends on $m-2$ parameters.

In case $E$ is a CM field, let $\sigma_0$ be an embedding of $E$ in $\CC$ for which
$(U_{\sigma_0}+U_{\bar{\sigma}_0})\cap U_{\RR}$ has signature $(2,2m-2)$.
The eigenspace $U_{\sigma_0}$ of $E$ is isotropic for $q$ and a very general $\omega\in \mathbb P U_{\sigma_0}$
defines a simple polarized Hodge structure on $U$ with $U^{2,0}=\CC\omega$.
In this case $\omega$ depends on $m-1$ parameters and
we obtain an $(m-1)$-dimensional family of HK manifolds whose very general member $X$ has an Hodge isometry
$U\simeq T_{X,\QQ}$.

In each case, we have $E\subset  {\rm End}_{\rm Hdg}(T_{X,\QQ})$ by construction,
and it remains to prove that this is an equality.

To this end, we consider first the case that $E$ is a CM field.

The HK manifolds  deform in a family of dimension $m-1$ such that
$E\subset  {\rm End}_{\rm Hdg}(T_{\tilde X})$
for any member $\tilde X$;
this should be understood to include the case $m=1$ where the period map has zero-dimensional fibres
above the given Hodge structure.
Picking a very general member $X_0$
of this family  (or one of the finitely many in the fiber of the period map in case $m=1$),
the claimed equality follows for $X_0$ for reason of moduli dimensions.
Namely, if
$E\subsetneq E' = {\rm End}_{\rm Hdg}(T_{X_0,\QQ})$,
then the degree $d'$ of $E'$ is a multiple of $d$ and with
$d'm'=dm$, we see that $m'\leq m/2$, so this can obviously not happen if $m=1$.
If $m>1$, then the family, which has CM by $E'$, has dimension at most $m/2-1<m-1$.
This contradicts the very general choice of $X_0$ and thus completes the proof of Theorem \ref{thm:Torelli}
in the CM case.

If $E$ is totally real of degree $d$
then, since $m\geq 3$, we have an $(m-2)$-dimensional family of HK manifolds $X$ with
$E\subset {\rm End}_{\rm Hdg}(T_{X,\QQ})$. 
Equality follows again from the dimension of the deformations:
if
$E\subsetneq E' = {\rm End}_{\rm Hdg}(T_{X_0,\QQ})$ for a very general $X_0$,
then the degree $d'$ of $E'$ is a multiple of $d$ and with
$d'm'=dm$ we see that $m'\leq m/2$ so the family which has CM or RM by $E'$ has dimension at most $m/2-1<m-2$,
contradicting the choice of $X_0$.

Finally, since $T_{X_0,\QQ}\cong U$, we infer from \eqref{eq:Pic(X)} that $\Pic(X_0)\cong U^\perp\cap L_{r,n}$ as stated.
\qed

\begin{remark}
\label{rem:End}
The above argument
remains valid for $m=2$ in the totally real case.
The only difference in that special case is
that the equality of $E$ with the algebra of Hodge endomorphisms from Theorem \ref{thm:Torelli} becomes a strict inclusion,
$E\subsetneq  {\rm End}_{\rm Hdg}(T_X)$,
since the latter automatically is a CM field, cf.\ the proof of Lemma \ref{lem:end}.
\end{remark}

\begin{remark}
In the $(m=1)$-case in Theorem \ref{thm:Torelli}
the  HK manifolds with CM are automatically defined over some number field. 
Indeed, if such an HK manifold were defined over some field of positive transcendence degree over $\QQ$,
then a transcendental element of this field would define a deformation of $X$.
Since this preserves ${\rm End}_{\rm Hdg}(T_X)$ by construction,
this contradicts the fact that the fibers of the period map are zero-dimensional.

Also the moduli spaces for higher $m$ may be of interest,
notably the Shimura curves for the RM case with $m=3$. 
\end{remark}

\begin{example}
\label{ex:Bert3}
The double sextic K3 surfaces with RM by real quadratic fields from Example \ref{ex:Bert2}
deform in one-dimensional families (a few of which have been worked out explicitly in \cite{EJ 14, EJ 16, EJ 23}).

By Lemma \ref{lem:end},
the members of these families with $\rho>16$ (so that $\rho=18$ or $20$)
automatically have CM (cf.\ \cite[Prop.\ 7.7]{GS} for a similar example).
\end{example}

\begin{remark}
There are also applications to Fano varieties with K3-type Hodge structures. 
For instance, the Fano variety of lines $F(X)$ on a cubic fourfold $X$ is a HK variety,
and by \cite{BD}
there is a correspondence on the product $X\times F(X)$
which induces an isomorphism on the Hodge structures. 
This translates RM or CM structures from one to the other.
\end{remark}

\section{Proof of Theorems A and B}\label{K3 section}
\label{s:pf}

The aim of this section is to prove the two main theorems stated in the introduction.

\subsection{Proof of Theorem B}
\label{pf2}

Given a CM field $E$ of degree $d$ over $\CC$ and an integer $m$ such that $md\leq 20$,
Proposition \ref{cm 2} provides a quadratic form $W$ over $E$ such that
${\TT}(W)$ has signature $(2,md-2)$ and embeds into $V_{K3}$.
Applying Theorem \ref{thm:Torelli} to $\TT(W)$ with its natural $E$-module structure,
we obtain an $(m-1)$-dimensional family of complex projective K3 surfaces
with $E=  {\rm End}_{\rm Hdg}(T_{X,\QQ})$  for a very general member $X$.

If $m=1$ and $\Delta_E$ is not a square, then Proposition \ref{m = 1}
combines with Theorem \ref{thm:Torelli} to produce in fact countably many 
complex projective K3 surfaces (each defined over some number field)
with the required properties.

Suppose that $m = 1$ and that  $\Delta_E$ is a square. Let $N$ be a positive integer; then $H(N)$,
the hyperbolic plane with intersection form scaled by $N$,
 can be primitively embedded in
 the $K3$-lattice $\Lambda_{K3}$. Let $U_{\bf Z}$ be the orthogonal complement of $H(N)$ in $\Lambda_{K3}$, and set $U = U_{\bf Z} \otimes_{\bf Z}{\QQ}$. Note that $H(N) \otimes_{\bf Z} \QQ$ is isomorphic to $H$, as quadratic forms over $\QQ$, therefore
 $U \oplus H \simeq V_{K3}$. This implies that $U \simeq H^2 \oplus I_{16}$, and hence $U \otimes_{\bf Q} {\bf Q}_p$ is 
 isomorphic to an orthogonal sum of hyperbolic planes for all prime numbers $p$. 
Since $\Delta_E$ is a square and ${\rm det}(U) = 1$, the
 quadratic form $U$ satisfies the hypotheses of Theorem \ref{realization}, hence there exists a one-dimensional hermitian
 form $W$ such that $U \simeq \TT(W)$. By Theorem  \ref{thm:Torelli} there exists a complex projective K3 surface $X$ with ${\rm Pic}(X) \simeq H(N)$ and
 $E=  {\rm End}_{\rm Hdg}(T_{X,\QQ})$. Since $N$ can take any positive integral value, we obtain infinitely many complex projective K3 surfaces
 with the required properties.
\qed

\begin{remark}
\label{rem:Tate}
The K3 surfaces constructed for $N>1$
may also be interpreted in terms of Tate--Shafarevich groups of the K3 surface with $N=1$, as in \cite[\S 11.4]{H},
or in terms of homogeneous spaces.
All these K3 surfaces happen to be elliptic, cf.\ also Proposition \ref{prop:g=1}.
\end{remark}

\subsection{Proof of Theorem A}
\label{pf1}

Let $E$ be totally real of degree $d$ and $m\geq 3$.
From Corollary \ref{for Theorem A} we obtain $W$ over $E$ such that
$V_{K3}=\TT(W)\oplus V'$ and the $dm$-dimensional quadratic form
$\TT(W)$, which satisfies condition (C) by construction,
also satisfies the other conditions in Theorem \ref{thm:Torelli}.
Hence there is an $(m-2)$-dimensional family of complex projective K3 surfaces with
$E= {\rm End}_{\rm Hdg}(T_{X,\QQ})$ for a very general member $X$. 
\qed

\begin{remark}
Theorem A (and B) `only' show that for any totally real (or CM) field $E$ there exists a K3 surface $\tilde{X}$,
with a transcendental lattice of a given rank, such that
$E={\rm End}_{\rm Hdg}(T_{\tilde X})$.
However, Proposition \ref{3,7} actually gives a much more precise result,
showing that the Picard lattice $\Pic(\tilde X)$, of rank one, can be chosen arbitrarily.

In fact, given a totally real field $E$ of degree 3 or 7, let $N$ be a positive even integer.
 Let $V'=\langle N\rangle$, then $V_{K3}\simeq \TT(W)\oplus V'$ for some quadratic form $W$ over $E$.
By Witt's extension theorem, we may assume that a primitive vector
$f\in \Lambda_{K3}\subset V_{K3}$ with $f^2=N$ maps to a generator of $V'$. Choosing a very general Hodge structure on $f^\perp$ we find a K3 surface $\tilde{X}$ of degree $N$ with $\Pic(\tilde{X})={\ZZ} f$ which has RM by $E$.

In the same spirit, the results of Section \ref{characterization section} allow for the characterization of Picard lattices
and transcendental lattices of K3 surfaces (and more generally, of HK manifolds,
see section \ref{s:HK}) with given endomorphism algebras. This is the aim of the next section.
\end{remark}

\section{Picard lattices}\label{Picard section}
\label{s:Pic}

We start with a lattice $L$ which is primitively embedded in the $K3$-lattice $\Lambda_{K3}$; set $\rho = {\rm rank}(L)$, and assume
that the signature of $L$ is $(1,\rho - 1)$.
Then $L^\perp$ has signature $(2,20-\rho)$ and after choosing a general Hodge structure of K3 type on $L^\perp$
we obtain a K3 surface $X$ with $T_X\simeq L^\perp$ and $\Pic(X)\simeq L$.
We then ask whether there are such K3 surfaces 
with RM or CM by a given field $E$.
To this end, we let $E$ be a totally real or CM number field of degree $d$
and we assume that there is an $m\in\NN$ such that $\rho + md = 22$.

\begin{theo}\label{odd} Suppose that $E$ is totally real and $m \geqslant 3$.
Assume moreover that  $d$ is odd
or ${\rm det}(L) \Delta_E \in \Lambda^+_{E/{\QQ}}$. 
Then there exists an $(m-2)$-dimensional family of complex projective $K3$ surfaces $X$
such that a very general member of this family satisfies

\begin{itemize}
\item[$\bullet$] 
$A_X \simeq E$;

\item[$\bullet$]
 ${\rm Pic}(X) \simeq L$. 
\end{itemize}
\end{theo} 

\noindent
{\bf Proof.} Note that $V':=L \otimes_{\bf Z} {\QQ}$ embeds into
$V_{K3} = \Lambda_{K3}  \otimes_{\bf Z} {\QQ}$. Set $V = V_{K3}$, and let
$U$ be the orthogonal complement of $V'$ in $V$, i.e.\ $U \oplus V' \simeq V$.
By Theorem \ref{new} (if $d$ is odd) and Theorem  \ref{even degree} (if $d$ is even)
there exists a quadratic form $W$ over $E$ such that $U = T(W)$;
moreover,  there is an embedding $\sigma \in \Sigma_E$ such
that the signature of $W_{\sigma}$ is $(2,m-2)$. The result for the quadratic form $V'=L\otimes\QQ$ 
thus follows from Theorem \ref{thm:Torelli}.
The statement for the precise given lattice structure then follows
from the given primitive embedding $L\hookrightarrow \Lambda_{K3}$, as pointed out in Section \ref{ss:back}.
\qed

\begin{remark}
The odd degree case of Theorem \ref{odd} exactly gives Theorem ${\mathrm A}'$ from the introduction.
\end{remark}

In the case of CM fields, the conditions are much more restrictive.

\begin{theo}\label{Pic cm} Suppose that $E$ is a CM field and that $m \geqslant 2$. There exists
 an $(m-1)$-dimensional family of complex projective $K3$ surfaces $X$
such that a very general member of this family satisfies

\begin{itemize}
\item[$\bullet$] 
$A_X \simeq E$;

\item[$\bullet$]  
${\rm Pic}(X) \simeq L$
\end{itemize} 
if and only if ${\rm disc}(L \otimes_{\bf Z} {\QQ}) = \Delta_E^m$, and $L\otimes_{\QQ}{\QQ}_p$ is isomorphic to an orthogonal sum of hyperbolic planes for all $p \in S_E$.

\end{theo}

\noindent
{\bf Proof.} 
The conditions are necessary by Theorem \ref{CM characterization}. Let us prove that they are sufficient.
As in the previous proof, let $V = V_{K3}$, and let
$U$ be the orthogonal complement of $V'=L \otimes_{\bf Z} {\QQ}$ in $V$, i.e.\ $U \oplus V' \simeq V$. Since ${\disc}(V) = 1$, we
have ${\disc}(U) = \Delta_E^m$. If $p \in S_E$, then 
$V' \otimes_{\QQ}{\QQ}_p$ is isomorphic to a sum of hyperbolic planes. 
This implies that $U \otimes_{\QQ} {\QQ}_p$ 
is also isomorphic to a sum of hyperbolic planes. 
The signature of $U$ is $(2,md-2)$. 
Applying Theorem \ref{CM characterization}
we conclude that there exists a hermitian form $W$ 
such that $U \simeq T(W)$. 
Then we can use Theorem \ref{thm:Torelli}
to conclude the existence of the desired family of K3 surfaces.
\qed

\subsection{Connection with elliptic fibrations}

In the remainder of this section we consider some examples and special cases related to elliptic fibrations.

\begin{prop}
\label{prop:g=1}
Let $X$ be a K3 surface with $\rho(X)=2$ such that $A_X$ contains a CM field of degree $2$ or $10$.
Then $X$ admits an elliptic fibration.
\end{prop}

\begin{proof}
Let $E\subset A_X$ be a CM field of degree $d=2$ or $10$, then $m=10$ or $2$ and hence $\Delta_E^m$ is a square.
Then by Proposition \ref{cm 2} we have 
$\Pic(X)\otimes\QQ = \TT(W)^\perp = H$.
In particular, $\Pic(X)$ represents zero.
It is well known that this implies that $X$ admits an elliptic fibration 
as stated
(i.e.\ almost all fibres are elliptic curves, but note that the fibration may not admit a section,
cf.\ \cite[Thm.\ 11.24]{MWL}).
\qed
\end{proof}

\medskip

We also record the following converse of Proposition \ref{prop:g=1} on the level of lattices:

\begin{lemma}
Let $X$ be an elliptic K3 surface with $\rho(X)=2$ and $E$ a CM field of degree $2$ or $10$.
Then $L=\NS(X)$ and $E$ satisfy the conditions of Theorem \ref{Pic cm}.
\end{lemma}

\begin{proof}
Since disc$(L)=1$ and $m$ is even,
this amounts to an easy check using
first the elementary fact that $L\simeq H$ over $\QQ$ (derived from the existence of a non-zero isotropic vector in $L$),
so $L^\perp \simeq H^2\oplus E_8^2$,
and second
the classification of $p$-adic quadratic forms to verify that $E_8\simeq H^4$ over $\QQ_p$ for any prime $p$.
\qed
\end{proof}

\begin{remark}
For $E$ imaginary quadratic, 
this validates the claim from \cite[Rem.\ 3.13]{GS},
even though the $E$-module structure on $L^\perp$  given there is not sufficient as it
is not compatible with the complex conjugation condition \eqref{eq:qq}.
\end{remark}

\begin{example}
\label{ex:non-sympl}
It was proved  in \cite{AS} that any K3 surface $X$ admitting a  non-symplectic automorphism of order $3$
(i.e.\ acting non-trivially on the holomorphic 2-form, such that $X$ has CM by $\QQ(\sqrt{-3})$)
has the lattice $H$ or $H(3)$ as an orthogonal summand of $\Pic(X)$,
in perfect agreement with Proposition \ref{prop:g=1}.

In contrast, if one considers a K3 surface $X$ admitting a  non-symplectic automorphism of order $5$,
this gives CM by $\QQ(\zeta_5)$ which is not covered by Proposition \ref{prop:g=1}.
Indeed, then the maximal family from
\cite{AST} has dimension $4$,
which corresponds to $m=5$ in Theorem \ref{thm:Torelli}.
A very general member has Picard lattice of rank two and determinant $-5$,
thus not representing zero and in particular not allowing for a genus one fibration.
\end{example}

\begin{prop}
\label{prop:m=1}
Let $X$ be a K3 surface with complex multiplication by a CM field $E$ of degree 20.
Then $\Delta_E$ is a square if and only if
 $X$ admits an elliptic fibration.
\end{prop}

\noindent
{\bf Proof.} 
Note that $\rho(X) = 2$ by assumption.
Let $\Delta_E$ be a square.
Then, by Proposition \ref{m = Delta = 1}, we have $\Pic(X)\otimes\QQ = \TT(W)^\perp = H$, 
hence $\Pic(X)$ represents zero.
As before, this implies that $X$ admits an elliptic fibration.

Conversely, if $\Delta_E$ is not a square,
then neither is $-\det(\Pic(X))$;
in particular, $\Pic(X)$ cannot represent zero,
hence $X$ does not have any elliptic fibration. 
\qed

\begin{remark}
This fits perfectly with the specific examples for the case $d=20$ 
constructed in the proof of Theorem B, see also Remark \ref{rem:Tate}.
\end{remark}

\begin{example}\label{Kondo} Kondo's examples of K3 surfaces with CM by degree 20 cyclotomic fields $E$ with $\Delta_E$ a square, namely ${\QQ}(\zeta_m)$ with
$m = 44$ or $66$, have elliptic fibrations (see \cite{Ko}, or \cite{LSY}, Table 2). On the other hand, Vorontsov's example having CM by 
${\QQ}(\zeta_{25})$ does not admit an elliptic fibration~ (see \cite{Ko}); the discriminant of ${\QQ}(\zeta_{25})$  is not a square, hence
the hypotheses of Proposition \ref{prop:m=1} are not fulfilled. 
\end{example}

\begin{prop}
Let $X$ be a K3 surface with complex multiplication by a CM field $E$ of degree 4.
Then
 $X$ admits an elliptic fibration
  if and only if $\rho(X)\geq 6$ or
  $\Delta_E$ is a square.
\end{prop}

\begin{proof}
If $\rho\geq 6$, then $\Pic(X)$ represents zero by Meyer's theorem.
Hence $X$ admits an elliptic fibration as before.

The $E$-module structure on $T_{X,\QQ}$ only leaves the case $m=\dim_E(T_{X,\QQ})=5$ 
whence $\rho(X)=2$.
 In this case, the argument is identical to the proof of Proposition \ref{prop:m=1}.
\qed
\end{proof}

\section{RM and CM for higher dimensional HK manifolds}\label{RMCMHK}

The following theorems identify the four known families of higher-dimensional HK manifolds
from Table \ref{table}  by their second Betti number $r=b_2$.

\subsection{HK with RM}
\label{HK+RM}

\begin{theo}
\label{thm:HK+RM}
Let $E$ be a totally real number field of degree $d$, let $r\in\{7,8,23,24\}$ and let $m$
be an integer with $m \geqslant 3$ and $md \leqslant r-1$.

Then there exists an $m-2$-dimensional family of projective HK manifolds with $b_2=r$,
of any dimension $n\geq 1$ in case $r=7,23$,
such that the very general member $X$ has the properties
$$
A_X \simeq E,\quad   \mathrm{dim}_E(T_{X,\QQ}) = m,\quad \mathrm{rk}(\Pic(X))=r-md~.
$$
\end{theo}

\begin{proof}
In case $r=7$, a $2n$-dimensional HK manifold of generalized Kummer type has
$H^2(X,\QQ)\simeq H^3\oplus \langle -2n-2\rangle$.
Proposition \ref{Kum and K3[n]} shows that if there is an isomorphism of lattices
$\Pic(X)\simeq \langle h\rangle$
for some $h\in 2\NN$ with $-2(n+1)h \in {\mathrm N}_{E/\QQ}(E^\times)$, 
then $T_X=\TT(W)$ for a quadratic form $W$ over $E$
with the required signature Condition (C) from Definition \ref{def:C}. 
As $\dim T_{X,\QQ}=6=2\cdot 3$, Theorem \ref{thm:Torelli} gives an $3-2=1$ dimensional family of
such manifolds with RM by a quadratic field $E$,
the very general member of which has the desired properties.

In case $r=8$,  since $X$ is assumed to be projective, $\dim T_{X,\QQ}\leq 8-1=7$. For RM by $E$ we need
$md\leq 7$ and $m\geq 3$ thus the only possibility is $d=2$, $m=3$ and then $\mathrm{rk}(\Pic(X))=2$.
The result follows from Corollary \ref{real} with
$V=H^2(X,\QQ)$ and $V'=\Pic(X)\otimes\QQ$,
combined with Theorem \ref{thm:Torelli}.

In case $r=23$ and $\Pic(X)= \langle h\rangle$, 
$T_{X,\QQ}$ has dimension $22$ and thus if $A_X$ is a totally real field $E$, it must have degree $2$ and
$\dim_E(T_{X,\QQ})=11$. Proposition \ref{Kum and K3[n]} shows that if $-2(n-1)h \in {\mathrm N}_{E/\QQ}(E^\times)$
then $T_X=\TT(W)$ for a quadratic form $W$ over $E$ with the required signature condition and we obtain an $11-2=9$-dimensional family of such manifolds with RM by the quadratic field $E$, again from Theorem \ref{thm:Torelli}.
In case $\mathrm{rk}(\Pic(X))>1$, the result follows again from the combination of  Corollary \ref{real} and Theorem \ref{thm:Torelli}.

In case $r=24$, $\dim T_{X,\QQ}\leq 23$ and to have RM one needs $\dim T_{X,\QQ}=dm$ for $m\geq 3$, hence
we actually only need to consider the case that $\dim T_{X,\QQ}\leq 22$. As then $\mathrm{rk}(\Pic(X))>1$,
the result again follows from Corollary \ref{real} and Theorem \ref{thm:Torelli}.
\qed
\end{proof}

\subsection{HK with CM}\label{HK+CM}

\begin{theo}
\label{thm:HK+CM}
Let $E$ be a CM number field of degree $d$, let $r\in\{7,8,23,24\}$ and let $m$
be an integer with $m \geqslant 1$ and $md \leqslant r-1$.

Then there exists an $(m-1)$-dimensional family of projective HK manifolds with $b_2=r$,
of any dimension $n\geq 1$ in case $r=7,23$,
such that the very general member $X$ has the properties
$$
A_X \simeq E,\quad   \mathrm{dim}_E(T_{X,\QQ}) = m,\quad \mathrm{rk}(\Pic(X))=r-md~.
$$
\end{theo}

\begin{proof}
In case $r=7$, since $X$ is assumed to be projective, we must have $md=\dim T_{X,\QQ}\leq b_2(X)-1=6$,
and the result follows from Proposition \ref{Kum and K3[n] CM} combined with Theorem \ref{thm:Torelli}.

In case $r=8$, since $d$ is even, we actually have $md\leq 6$ and the result follows from Corollary \ref{OG6}
and Theorem \ref{thm:Torelli}.

In case $r=23$, since we assume $X$ to be projective, we have $md=\dim T_{X,\QQ}\leq 22$, and
the result again follows from Proposition \ref{Kum and K3[n] CM} combined with Theorem \ref{thm:Torelli}.

In case $r=24$, since $X$ is projective and $d$ is even,
we actually only need to consider the case that $\dim T_{X,\QQ}\leq 22$.
The result then follows from Corollary \ref{OG10} in combination with  Theorem \ref{thm:Torelli}.
\qed
\end{proof}

\smallskip

\begin{remark}
A (very) particular case of a HK manifold $X$ with CM is when $X$ has an automorphism $\phi:X\rightarrow X$
of order $p$ such that $\phi^*$ induces multiplication by a primitive $p$-th root of unity on $H^{2,0}(X)$ for
a prime $p>2$
(like for K3 surfaces in Example \ref{ex:non-sympl}).
In that case $T_X$ has CM by the cyclotomic field $\QQ(\zeta_p)$ of $p$-th roots of unity.
For recent results on such automorphisms see \cite{BC} and the references given there.
\end{remark}

\subsection{The Hodge group and the Mumford--Tate group}
\label{ss:MT}
Zarhin showed that the Hodge group, also known as the
Special Mumford-Tate group, of a simple rational Hodge structure K3 type $V$
is determined by its endomorphism algebra $E$,
recall that there is then a form $W$ over $E$ (quadratic or hermitian) such that there is an isometry $V\simeq \TT(W)$.
In fact, the Hodge group is the subgroup of
$E$-linear special isometries ${\rm SO}_E(V)\subset SO(V)$ \cite[Thm.\ 2.2.1, Thm.\ 2.3.1]{Z},
cf.\ \cite{H}, Chap. 3, \S3, especially Theorem 3.9.

It is natural to ask which groups actually occur.
The following result shows that if $md < 20$, 
every such special orthogonal and unitary group occurs for some  K3 surface.
To introduce the necessary notation, note that
the group ${\rm SO}$ is an algebraic group defined over the field $E$, 
and SU is defined over $E_0$. 
To obtain the corresponding Mumford-Tate group, 
we need to consider these groups as defined over $\QQ$; 
this is achieved by taking the restriction of scalars, denoted by Res$_{E/\QQ}$.

\begin{theo} 
\label{thm:Hodge_group}
Let $E$ be a totally real or CM field of degree $d$ and let $m \geqslant 1$ be an integer such that $md <20$. Let
$W$ be a quadratic (respectively hermitian) form of dimension $d$ over $E$
such that $\TT(W)$ has signature $(2,md-2)$. 
If $W$ is quadratic, assume in addition that
$m \geqslant 3$ and that there exists a real embedding of $E$ such that $W_\sigma$ has signature $(2,m-2)$.
Then there exists a complex projective $K3$ surface with
Hodge group isomorphic to ${\rm Res}_{E/{\QQ}}({\rm SO}(W))$ if $E$ is totally real and
${\rm Res}_{E/{\QQ}}({\rm U}(W))$ if $E$ is a CM field.
\end{theo}

\noindent
\begin{proof} This follows from Theorem \ref{cm Hodge} if $E$ is CM, and from Theorem \ref{Hodge} if $E$ is totally real,
combined with Theorem \ref{thm:Torelli}. 
\qed
\end{proof}

\begin{remark}
Note that for any field $E$ which is CM of degree $\leq 20$, or totally real of degree $\leq 7$,
there exists some $W$ as in Theorem \ref{thm:Hodge_group}
by Proposition \ref{cm 2} resp.\ by Corollary \ref{real}.
\end{remark}

\subsection*{Acknowledgements}
We thank Simon Brandhorst and John Voight for very helpful comments.
We are very grateful to the referee for many useful comments and corrections.

\end{document}